\documentclass{amsart}

\usepackage{amssymb}
\usepackage{amsmath}
\usepackage{amscd}
\usepackage{amsthm}
\usepackage{latexsym}

\newtheorem{thm}{Theorem}[section]
\newtheorem{prop}[thm]{Proposition}
\newtheorem{lem}[thm]{Lemma}
\newtheorem{cor}[thm]{Corollary}
\newtheorem{rem}[thm]{Remark}
\newtheorem{defin}[thm]{Definition}

\newcommand{\ot}{\otimes}
\newcommand{\op}{\oplus}

\newcommand{\lorarr}{\longrightarrow}

\newcommand{\isoarr}{\stackrel{\sim}{\longrightarrow}}
\newcommand{\thrarr}{\twoheadrightarrow}

\newcommand{\vphi}{\varphi}

\newcommand{\ra}{\rangle}
\newcommand{\la}{\langle}

\newcommand{\bZ}{\mathbb{Z}}

\newcommand{\bF}{\mathbb{F}}

\newcommand{\bDA}{\mathbb{DA}}
\newcommand{\bCA}{\mathbb{CA}}

\newcommand{\Fp}{\mathbb{F}_p}

\newcommand{\Hom}{\mathrm{Hom}}

\newcommand{\End}{\mathrm{End}}

\newcommand{\Out}{\mathrm{Out}}
\newcommand{\res}{\mathrm{res}}

\newcommand{\rad}{\mathrm{rad}}
\newcommand{\Tr}{\mathrm{Tr}}

\newcommand{\GL}{\mathrm{GL}}

\DeclareMathOperator{\Img}{Im}
\DeclareMathOperator{\Ker}{Ker}

\newcommand{\mcl}{\mathcal}

\newcommand{\mmod}{\mathrm{mod}}
\newcommand{\MMod}{\mathrm{Mod}}

\newcommand{\lam}{\lambda}
\newcommand{\Lam}{\Lambda}
\newcommand{\al}{\alpha}

\title[Cohomology of Extraspecial $p$-Groups]
{Representations of the Double Burnside Algebra 
and Cohomology of the Extraspecial $p$-Group}
\author{Akihiko Hida and Nobuaki Yagita}

\address{Akhiko Hida, Faculty of Education, 
Saitama University,
Shimo-okubo 255, Sakura-ku, Saitama-city, Saitama, Japan}
\email{ahida@mail.saitama-u.ac.jp}

\address{Nobuaki Yagita, 
Department of Mathematics, Faculty of Education, 
Ibaraki University,
Mito, Ibaraki, Japan}
\email{yagita@mx.ibaraki.ac.jp}



\begin{document}

\begin{abstract}
Let $E$ be the extraspecial $p$-group of order $p^3$ and exponent 
$p$ where $p$ is an odd prime. We determine the mod $p$  cohomology of summands in the stable splitting of 
$p$-completed classifying space $BE$ modulo nilpotence. It is 
well known that indecomposable 
summands in the complete stable splitting correspond to  
simple modules for the mod $p$ double Burnside algebra. 
We shall use representation 
theory of the double Burnside algebra and the 
theory of biset functors. 
\end{abstract}

\maketitle

\section{Introduction}
Let $p$ be an odd prime and $E=p_+^{1+2}$ the extraspecial 
$p$-group of order $p^3$ and exponent $p$.  The (integral or 
mod $p$) cohomology ring of $E$ is widely studied (for example,  
\cite{Gr}, \cite{L91}, \cite{L92}, \cite{Lewis}, \cite{S}, \cite{TY}, \cite{Y98}, \cite{Y07}).  In particular, in \cite{Y07}, the second author consider 
cohomology rings of finite groups with Sylow 
$p$-subgroup $E$, or more generally, cohomology rings 
of saturated fusion systems on $E$, and obtained 
the cohomology of various summands in the stable splitting of 
$BE$ where $BE$ is a $p$-completed classifying space of 
$E$.  

Let $P$ be a finite $p$-group where $p$ is an arbitrary prime.  
By the Segal conjecture (Carlsson's theorem \cite{C}) and the 
result of Lewis, May and McClure \cite{LMM}, it is well known that, 
in the stable homotopy category, the endomorphism ring 
of $BE$ is isomorphic to the completion of the double Burnside 
ring $A_{\bZ}(P,P)$. 
Consequently, indecomposable stable summands of 
$BP$ correspond to primitive idempotents of the double 
Burnside algebra $A_p(P,P)$ over a finite field $\bF_p$, or equivalently, simple $A_p(P,P)$-modules (except the one dimensional module with 
trivial minimal subgroup).  Benson and Feshbach \cite{BF} and 
Martino and Priddy \cite{MP} classified simple 
$A_p(P,P)$-modules and studied stable splittings of $BP$ 
for various $p$-groups. See Benson's survey \cite{B} for 
more details. 

In this paper, we shall determine the cohomology of stable 
summands of $BE$ for the extraspecial $p$-group $E=p_+^{1+2}$ 
through the action of double Burnside algebra 
$A_p(E,E)$ over  $\bF_p$. Let 
$$BE=\bigvee_i\bigvee_{1\leq j \leq m(i)} X_{ij}$$
be the complete stable splitting of $BE$ such 
that $X_{ij} \sim X_{mn}$ if and only if $i=m$. 
Let 
$$1=\sum_i\sum_{1 \leq j \leq m(i)}e_{ij}$$
be a corresponding decomposition of unity into 
orthogonal primitive idempotents in $A_p(E,E)$. 
Consider the sum $X_i=\vee_{1 \leq j \leq m(i)} X_{ij}$ 
of all equivalent indecomposable summands. 
Let $e_i=\sum_{1 \leq j \leq m(i)}e_{ij}$. 
We want to consider the cohomology of $X_i$ modulo nilpotence. 
However, the cohomology of $X_i$ does not have ring structure which relates to  
the structure of the cohomology ring of $BE$. Hence, in this 
paper, let 
$$H^*(X_i)=H^*(E)e_i$$
where
$$H^*(E)=(\bF_p\ot H^*(BE;\bZ))/\sqrt{0}.$$
Note that $A_p(E,E)$ acts on $H^*(E)$ (see Lemma \ref{l34} below) and $H^*(X_i)$ is actually  defined, though it depends on the choice of an idempotent $e_i$. 
The structure of $H^*(E)$ is known (\cite{L91}, \cite{TY}) and 
the main purpose of this paper is to determine   
$H^*(X_i)$ for every summand $X_i$, see section \ref{main} below. 

For example, let $X(\bF_p)$ be the principal dominant summand of 
$BE$, namely, the summand corresponding to the trivial 
$\Out(E)$-module. 
In \cite{Y07}, it was shown that $H^*(X(\bF_p))$ is the Dickson subalgebra 
of $H^*(E)$ for $p=3,5,7$.  
These are obtained by the cohomology 
of the fourth Janko group $J_4$ ($p=3$)(Green \cite{Gr}), the Thompson group $Th$ ($p=5$) and the exotic saturated fusion 
systems of Ruiz and Viruel \cite{RV} ($p=7$). 
On the other hand, in this paper, 
we prove this result for every $p$ without 
use of large sporadic simple groups.  
\vspace{.2cm}\\
{\bf Corollary \ref{c107}.}  Let $X(\bF_p)$ be the principal dominant 
summand of $BE$. Then 
$H^*(X(\bF_p))$ is isomorphic to the (positive degree 
part of)  Dickson subalgebra $\bDA$ of $H^*(E)$. 
\vspace{.2cm}

Let $S$ be a simple $A_p(E,E)$-module corresponding to 
$X_{ij}$. The multiplicity $m(i)$ is equal to the dimension 
of $S$ over $\bF_p$. Note that $\bF_p$ is a splitting field 
of $A_p(E,E)$. Dietz and Priddy \cite{DP} studied 
the stable splitting of $BE$, 
and in particular,  determined the multiplicity $m(i)$. 
Here, we reprove this result using cohomology, 
see Proposition \ref{p101}.  
 
Note that it may be possible to consider mod $p$ cohomology 
$H^*(BE;\bF_p)/\sqrt{0}$ instead of 
$H^*(E)=(\bF_p\ot H^*(BE;\bZ))/\sqrt{0}$. 
In general, $H^*(E)$ is a subalgebra of 
$H^*(BE;\bF_p)/\sqrt{0}$ and coincide with $H^*(BE;\bF_p)/\sqrt{0}$ if $p>3$. But for 
$p=3$, $H^*(E)$ is strictly smaller than $H^*(BE;\bF_p)/\sqrt{0}$. 
Here, we consider only $H^*(E)$ because the structure of 
$H^*(E)$ is easier than $H^*(BE;\bF_p)/\sqrt{0}$ for 
$p=3$ and we can treat all odd primes $p$ uniformly. 
Moreover, the $A_p(E,E)$-modules structure of $H^*(BE;\Fp)$ 
can be deduced from that of $H^*(E)$ by 
\cite{Y98}, \cite{Y07}.

Let $P$ be a finite $p$-group. Let $A_p(P,P)$ be the double 
Burnside algebra over $\bF_p$. Simple $A_p(P,P)$-modules are 
classified by \cite{BF}, \cite{MP}.  
In this paper, we shall follow the functorial approach by 
Webb \cite{W} and 
Bouc, Stancu and Th\'{e}venaz \cite{BST}.  In section \ref{biset}, we review some results in \cite{BST}. 

In section \ref{HE}, we summarize the results on the structure 
of $H^*(E)$. In section \ref{HA}, we study the cohomology of 
a maximal elementary $p$-subgroup $A$ of $E$. 
We compare the $\GL_2(\bF_p)$-module structures of 
$H^*(A)$ and $H^*(E)$ using  
the morphism $q^*:H^*(A) \lorarr H^*(E)$ induced by a surjective 
morphism $q:E\lorarr A$. Note that $\Out(E)\simeq \Out(A) 
\simeq \GL_2(\bF_p)$.   

In section \ref{transfer}, we calculate the image of 
the transfer map 
from the maximal elementary abelian $p$-subgroup $A$ of $E$. 
In section \ref{sub}, we study some $A_p(E,E)$-submodules 
of $H^*(E)$. In section \ref{Q}, we consider simple $A_p(E,E)$-modules with 
cyclic minimal subgroup. Similarly, in section \ref{A}, we 
consider simple $A_p(E,E)$-modules with minimal subgroup $A$. 

Using these results, in section \ref{main}, we determine 
$H^*(X_i)=H^*(E)e_i$ for every $X_i$. More precisely, 
we construct an $\Fp$-subspace $W$ of $H^*(E)$ such that 
the multiplication by $e_i$ induces an isomorphism 
$W \isoarr We_i=H^*(E)e_i=H^*(X_i)$. 
Note that, in general,  neither $W$ nor 
$H^*(X_i)$ is an $A_p(E,E)$-submodule of $H^*(E)$.


\section{Finite dimensional algebras and modules}\label{algebra}

Let $\Lam$ be a finite dimensional algebra over a filed $k$. 
We denote the Jacobson radical of $\Lam$ by $\rad(\Lam)$, 
namely, $\rad(\Lam)$ is the intersection of all maximal 
ideals of $\Lam$. 
If $e$ is a primitive idempotent in $\Lam$, then 
$e\Lam$ is a projective indecomposable right $\Lam$-module and 
$e\Lam /e\rad(\Lam)$ is a simple right $\Lam$-module. Let 
$$1=\sum_{1 \leq i \leq l}\sum_{1\leq j \leq m(i)} e_{ij}$$
be a decomposition of unity into primitive orthogonal 
idempotents in $\Lam$, where 
$$e_{ij}\Lam /e_{ij}\rad(\Lam)  \simeq e_{mn}\Lam /e_{mn}\rad(\Lam)$$ 
if and only of $i=m$. Then 
$$S_i=e_{i1}\Lam /e_{i1}\rad(\Lam)~(1 \leq i \leq l)$$
gives the complete set of representatives of isomorphism 
classes of simple right $\Lam$-modules. 
Let $e_i=\sum_{1\leq j \leq m(i)} e_{ij}$ and, in this paper, 
we call $e_i$ an  idempotent 
corresponding to the simple module $S_i$. Multiplication 
by $e_i$ induces the identity map on $S_i$ and 
$S_je_i=0$ for $j \ne i$. On the other hand, we have 
$$S_ie_{i1} \simeq \Hom_{\Lam}(e_{i1}\Lam, S_i) 
\simeq \End_{\Lam}(S_i)$$  
and the multiplicity $m(i)$ is equal to 
$\dim_{\End_{\Lam}(S_i)} S_i$. 

Let $M$ be a finite dimensional right $\Lam$-module. Then 
$M$ has a composition series  
$$0=M_0 \subset M_1 \subset \cdots \subset M_n=M$$ 
such that each quotient $M_j/M_{j-1}$ is a simple $\Lam$-module.  
The number of quotients $M_j/M_{j-1}$ such that 
$M_j/M_{j-1}\simeq S_i$  
equals to 
$$\dim_k Me_{i1}/\dim_k \End_{\Lam}(S_i)=
\dim_k  Me_i/m(i) \dim_k\End_{\Lam}(S_i)=
\dim_k  Me_i/ \dim_k S_i.$$

If every composition factor of $M$ is isomorphic to 
$S_i$, then $M=Me_i$ and the right multiplication by $e_i$ induces the identity on $M$.  In general case, to determine $Me_i$, we shall use the following lemma in section \ref{main}. Note that 
$H^*(E)$ is not finite dimensional, but each homogeneous part 
$H^n(E)$ is a finite dimensional $A_p(E,E)$-module and we can 
apply the lemma.

\begin{lem}\label{l21}
Let $M$ be a finite dimensional $\Lam$-module 
and $S$ a simple $\Lam$-module. 
Let $e$ be an idempotent corresponding to $S$. Let 
$$0 \subset L \subset N \subset M$$
be a sequence of $\Lam$-submodules. Suppose 
that $M/N$ has no composition factor which is 
isomorphic to $S$ and every simple subqutiont module of 
$N/L$ is isomorphic to $S$. If $W$ is a $k$-subspace of  
$N$ such that $N=L \op W$, then 
the multiplication by $e$ induces an isomorphism 
$$e: W\op Le \simeq We\op Le=Me.$$ 
\end{lem}

\begin{proof}
Since $(M/N)e=0$, we have $Me=Ne=We+Le$. Hence the map 
$e :W\op Le \lorarr Me$
is surjective. On the other hand, since 
$(N/L)e=N/L$, this map is injective and 
$Me=We \op Le$. 
\end{proof}


\section{Biset category and biset functor}\label{biset}
In this section, we summarize the theory of biset category and  functors in \cite{BST}. 
Here, we concentrate on right free bisets only, though,  
in \cite{BST}, Bouc, Stancu and Th\'evenaz consider the general bisets.

Let $G$ and $H$ be finite groups and let $A_{\bZ}(G,H)$ be 
the Grothendieck group of the category of the finite right free $(G,H)$-biset. 
Let $K\leq G$ and let $\vphi: K\lorarr H$ be a group 
homomorphism. We set 
$$G \times_{(K,\vphi)}H=G\times H/\sim$$
where $(gk,h)\sim (g, \vphi(k)h)$ for $g\in G$, $k \in K$ and 
$h \in H$. Then every transitive right free $(G,H)$-set is 
isomorphic to $G \times_{(K,\vphi)}H$ for some $K$ and 
$\vphi:K\lorarr H$. Let $\zeta_{K,\vphi}$ be the element of 
$A_{\bZ}(G,H)$ corresponding to $G \times_{(K,\vphi)}H$. 
Thus $A_{\bZ}(G,H)$ is a free abelian group with one basis 
element for each conjugacy class of the pair $(K,\vphi)$ 
where $K \leq G$ and $\vphi:K\lorarr H$. 

There exists a product 
$$A_{\bZ}(G_2,G_3) \times A_{\bZ}(G_1,G_2) \lorarr 
A_{\bZ}(G_1,G_3)$$
induced by the map
$$(X,Y) \mapsto Y\times_{G_2} X.$$
If $H_1\leq G_1$, $H_2 \leq G_2$, $\phi_1:H_1 \lorarr 
G_2$, $\phi_2:H_2 \lorarr G_3$, then the product is given by 
$$\zeta_{H_2,\phi_2}\zeta_{H_1,\phi_1}=
\sum_{x \in \phi_1(H_1)\backslash G_2/H_2}
\zeta_{\phi_1^{-1}(\phi_1(H_1)\cap {^xH_2}),\phi_2c_{x^{-1}}\phi_1}$$
(see \cite[p.160]{BF}).

Let $k$ be a field. We set 
$A_k(G,H)=k \ot A_{\bZ}(G,H)$.  
Then $A_k(G,G)$ is a finite dimensional $k$-algebra. 
We denote by $J(G)$ the $k$-subspace of $A_k(G,G)$ spanned by all $\zeta_{H,\vphi}$ such that $\vphi(H)<G$. 
Then $J(G)$ is an ideal of $A_k(G,G)$. 

There exists an injective $k$-algebra homomorphism 
$$\iota : k\Out(G) \lorarr A_k(G,G)$$
given by 
$\vphi \mapsto \zeta_{G,\vphi}$ where 
we denote the automorphism of $G$ which represents 
$\vphi$ by the same notation $\vphi$. 
Note that $\vphi_1, \vphi_2 \in \Out(G)$ then 
$\zeta_{G,\vphi_1}\zeta_{G,\vphi_2}=\zeta_{G,\vphi_1\vphi_2}$ and 
in fact $\iota$ is an algebra homomorphism.  
On the other hand, there exists a surjective $k$-algebra 
homomorphism 
$$\pi:A_k(G,G) \lorarr k\Out(G)$$
such that $\Ker \pi =J(G)$ and $\pi \iota$ is the identity 
map on $k\Out(G)$. 
We view 
$k\Out(G)$-modules as $A_k(G,G)$-modules via $\pi$. 
On the other hand, if $W$ is an $A_k(G,G)$-module, then 
we can view $W$ as a $k\Out(G)$-module via $\iota$. 

Let $\mcl{C}_k$ be the $k$-linear category whose object are 
all finite groups and morphisms are given by $\Hom_{\mcl{C}_k}(G,H)=A_k(G,H)$. A contravariant $k$-linear functor 
$$\mcl{C}_k \lorarr \mbox{$k$-$\MMod$}$$
is called an inflation functor \cite{W}. 
 
Let $F$ be an inflation functor. If we define the right action 
of $A_k(H,H)$ on $F(H)$ by 
$$v \phi =F(\phi)(v)$$
for $\phi \in A_k(H,H)$ and $v \in F(H)$, then $F(H)$ is a right $A_k(H,H)$-module since 
$F$ is a contravariant functor. 
If $V$ a $k$-subspace 
of $F(H)$, we set  
$$VA_k(G,H)=\sum_{\phi \in A_k(G,H)}F(\phi)(V)$$
for a finite group $G$,  
then 
$VA_k(-,H)$ is a subfunctor of $F$. 

Let $W$ be a simple $A_k(G,G)$-module. 
Then the inflation functor 
$$L_{G,W}=W\otimes_{A_k(G,G)}A_k(-,G).$$
has the unique maximal subfunctor  $J_{G,W}$ 
and
$$S_{G,W}=L_{G,W}/J_{G,W}$$
is a simple functor, where 
$$J_{G,W}(K)=\bigcap_{\phi \in A_k(G,K)}\Ker 
(L_{G,W}(\phi):L_{G,W}(K) \lorarr L_{G,W}(G))$$
by \cite[2.3. Lemma]{BST}. 

Let $S$ be a simple inflation functor. A minimal group of $S$ is a  group $H$ of minimal order such that $S(H) \ne 0$. 
If $H$ is a minimal subgroup of $S$ and $V=S(H)$, then 
$V$ is a simple $A_k(H,H)$-module by \cite[3.1. Proposition]{BST} 
and $S \simeq S_{H,V}$ by \cite[3.2. Proposition]{BST}.  
Let $P$ be a finite group. If $S_{H,V}(P)\ne 0$, 
then $S_{H,V}(P)$ is a simple $A_k(P,P)$-module by \cite[3.1. Proposition]{BST}.  
In this paper, we write 
$$S(P,H,V)=S_{H,V}(P)$$
if $S_{H,V}(P)\ne 0$ 
in order to emphasize that this is a nonzero simple 
$A_k(P,P)$-module. 
Conversely, let $W$ be a simple $A_k(P,P)$-module. 
The minimal group of $W$ is a group $H$ of minimal order 
such that 
$$WA_k(H,P)A_k(P,H) \ne 0.$$
If $H$ is a minimal subgroup of $W$ then there exists a simple $k\Out(H)$-module $V$ such that $S_{P,W}\simeq S_{H,V}$ and 
$$W\simeq S_{V,H}(P)=S(P,H,V)$$
as $A_k(P,P)$-modules by \cite[5.1. Proposition]{BST}. 
The following lemmas are almost same 
as \cite[3.5. Proposition]{BST}. We include proofs for 
completeness. 

\begin{lem}\label{l31} Let $F$ be an inflation functor. 
Let $H$ be a subgroup of $P$ and $V$ a simple 
$A_k(H,H)$-module. Suppose that 
$W=S_{H,V}(P)\ne 0$.  
Suppose that $V_2 \subset V_1$ are $A_k(H,H)$-submodules 
of $F(H)$  
and $V_1/V_2 \simeq V$ as $A_k(H,H)$-modules. 
Let $F_i=V_iA_k(-,H)$ be the subfunctor of $F$ generated by 
$V_i$. 
Then 
$M=F_1/F_2$ has the unique maximal subfunctor 
$M'$ such that $M/M' \simeq S_{H,V}$ where 
$$M'(X)=\bigcap_{\phi \in A_k(H,X)}\Ker M(\phi).$$
In particular, $M(P)/M'(P)\simeq W$. 
\end{lem}

\begin{proof} Since we have an isomorphism, 
$$\Hom (L_{H,V},F_1/F_2)\simeq 
\Hom_{A_k(H,H)}(V, F_1(H)/F_2(H))$$
by \cite[2.2. Lemma]{BST} or section 2 in \cite{Bo}, there exists a nonzero morphism 
$$\theta: L_{H,V} \lorarr M.$$
We define two subfunctors  $F_1'$ and $F_2'$ of $F$ by 
$$F'_1/F_2=\theta(L_{H,V}),~F_2'/F_2=\theta(J_{H,V}).$$
Then $F_2'/F_2$ is a unique maximal subfunctor of $F_1'/F_2$ and 
$$F_1'/F_2' \simeq L_{H,V}/J_{H,V}\simeq S_{H,V}.$$
Since $(F_1/F_1')(H)=0$ and $F_1(H)=V_1$, we have 
$F_1'(H)=V_1$. Hence we have that $F_1=F_1'$ since 
$F_1$ is generated by $V_1$. Hence 
$M'=F_2'/F_2$ is a unique maximal subfunctor of $M$. 
The description of $M'$ follows from the following commutative 
diagram: 
$$\begin{CD}
L_{H,V}(X)&@>{\theta(X)}>>&M(X) \\
@V{L_{H,V}(\phi)}VV &             &@VV{M(\phi)}V\\
L_{H,V}(H)& @>{\simeq}>>& M(H).
\end{CD}$$
\end{proof}

\begin{lem}\label{l32} Let $F$ be an inflation functor. 
Let $W$ be a simple $A_k(P,P)$-module with minimal subgroup 
$H$. Let $V$ be a simple $k\Out(H)$-module such that 
$W=S_{H,V}(P)$. 
Suppose that 
$W_2 \subset W_1$ are $A_k(P,P)$-submodules of $F(P)$ such 
that $W_1/W_2 \simeq W$. Let $V_1=W_1A_k(H,P)$. 
Then there exists an 
$A_k(H,H)$-submodules $V_2$ of $V_1$ such that 
$V_1/V_2 \simeq V$ as $A_k(H,H)$-modules and 
$$V_1A_k(P,H)/(V_1A_k(P,H)\cap W_2) \simeq W$$
as $A_k(P,P)$-modules.
\end{lem}

\begin{proof}
By Lemma \ref{l31}, $W_1A_k(-,P)$ has a maximal subfunctor $E'$ containing the subfunctor $W_2A_k(-,P)$ such that 
$$W_1A_k(-,P)/E' \simeq S_{P,W} \simeq S_{H,V}.$$
Let $V_2=E'(H)$. Then 
$V_1/V_2 \simeq S_{H,V}(H) \simeq V$. 
Since $W \simeq W_1/W_2$ and 
$WA_k(H,P)A_k(P,H)=W$, it follows that 
$$V_1A_k(P,H)+W_2=W_1A_k(H,P)A_k(P,H)+W_2=W_1.$$
Hence we have
$$V_1A_k(P,H)/(V_1A_k(P,H)\cap W_2) \simeq W_1/W_2 \simeq W.$$
This completes the proof. 
\end{proof}

\begin{cor}\label{c33} Let $F$ be an inflation functor. 
Let $H$ be a subgroup of $P$ and $V$ a simple 
$k\Out(H)$-module. Suppose that 
$W=S_{H,V}(P)\ne 0$. Let $N$ be a submodule of $F(H)$. 
If $F(H)/N$ has no
subquotient module isomorphic to $V$, then 
$F(P)/NA_k(P,H)$ has no subquotient module isomorphic to 
$W$.  
\end{cor}

\begin{proof}
Let $\bar{F}=F/NA_k(-,H)$. If $F(P)/NA_k(P,H)=\bar{F}(P)$ has 
subquotionet module which is isomorphic to $W$, then 
$\bar{F}(H)=F(H)/N$ has a subquotient which is isomorphic 
to $V$ by Lemma \ref{l32}. This contradicts the assumption.  
\end{proof}  

Now assume that $p$ is a prime number and $k=\bF_p$. 
We set 
$$A_p(P,Q)=A_{\Fp}(P,Q)=\Fp\ot A_{\bZ}(P,Q)$$ 
for finite group $P,Q$. 
Let $H^*(P)=(\bF_p \ot H^*(BP;\bZ) )/\sqrt{0}$. Let $H \leq P$ and let $\vphi : H \lorarr Q$ be a 
group morphism.  
The action of  $\zeta_{H,\vphi}\in A_p(P,Q)$ on cohomology is 
given by 
$$\Tr_H^P \vphi^*:H^*(BQ;\bF_p) \lorarr H^*(BP;\bF_p)$$
and $H(B(-);\bF_p)$ is an inflation functor. Similarly, 
$\bF_p\ot H(B(-);\bZ)$  
is an inflation functor. Moreover $H^*(-)$ is an inflation 
functor by the following lemma. 

\begin{lem}\label{l34}
Let $Q$ be a subgroup of $P$. If $x \in H^*(BQ;\bF_p)$ is nilpotent, 
then $\Tr_Q^P(x)$ is nilpotent.   
\end{lem}

\begin{proof}
We will prove $y=\Tr_Q^P(x)$ is nilpotent. 
By Quillen's theorem \cite{Q}, we only need to prove
that $\res^P_A(y)$ is nilpotent for each elementary
abelian $p$-subgroup $A$.
By the double coset formula, $\res^P_A(y)$ is a sum of
\begin{itemize}
\item[(i)] elements in the image of transfer map from the proper subgroup of $A$.
\item[(ii)] elements of the form $\res^P_A(x^t)$ where 
$x^t \in H^*(BQ^t, \bF_p)$ is the $t$-conjugate of $y$ for some 
$t \in P$ such that $A\subset Q^t$.
\end{itemize}
Since $A$ is elementary abelian, elements in (i) are zero.
Since conjugation and restriction maps are of course 
algebra morphisms, elements in (ii) are also nilpotent.
\end{proof}


\section{Cohomology of  $E=p^{1+2}$}\label{HE}
Let $p$ be an odd prime. 
The extraspecial $p$-group $E=p_+^{1+2}$ has a presentation as
\[ E= \la a,b,c~|~[a,b]=c,~ a^p=b^p=c^p=[a,c]=[b,c]=1 \ra .\]
The cohomology of $E$ is well known.
In particular (\cite{L91}, \cite{TY}, \cite{Y07}), 
$H^*(E)$ is generated by  
$$y_1,~y_2,~C,~v$$
subject to the following relation:
$$y_1^py_2-y_1y_2^p=0,~ Cy_i=y_i^{p},~
C^2=y_1^{2p-2}+y_2^{2p-2}-y_1^{p-1}y_2^{p-1}$$
where $|y_i|=2$, $|C|=2p-2$ and $|v|=2p$.  
We write $v^{p-1}$ by $V$.

Let $R$ be a subalgebra of $H^*(E)$ and 
$x_1, \dots, x_r$ elements of $H^*(E)$. 
We set 
$$R\{x_1, \dots, x_r\}=\sum_{i=1}^r Rx_i$$
if $x_1, \dots, x_r$ are linear independent over $R$. 
Moreover, if $W=\sum_{i=1}^r\bF_p x_i$ is a $\bF_p$-vector 
space spanned by $x_1,\dots, x_r$, then 
we set 
$$R\{W\}=R\{x_1, \dots, x_r\}.$$

\begin{lem}[{\cite[section 3]{Y07}}]\label{l41}
 We have the following expression,  
\[H^*(E)= \Fp[C,v]\{y_1^iy_2^j|0\le i,j\le p-1,
  (i,j)\not =(p-1,p-1)\}.\]
\end{lem}

The maximal elementary abelian 
$p$-subgroups of $E$ are 
\[A_i={\langle}c,ab^i{\rangle}\ \mbox{for} \ 0\le i\le p-1,\quad A_{\infty}={\langle}c,b{\rangle}.\]
Let $\mcl{A}(E)$ be the set of all elementary abelian $p$-subgroup 
of $E$, thus, 
$$\mcl{A}(E)=\{A_0, \dots, A_{p-1}, A_{\infty}\}.$$
Letting $H^*(A)\cong \Fp[y,u]$ and writing $i_{A}^*(x)=x|A$
for the inclusion $i_{A}{\colon}A \subset E$, the images of restriction maps are given by
\[y_1|A_i=y\ \mbox{for} \ i\in \bF_p,\ y_1|A_{\infty}=0,
\quad y_2|A_i=iy\ \mbox{for} \ i\in \bF_p,\ y_2|A_{\infty}=y,\]  
  \[C|A_i=y^{p-1},\quad v|A_i=u^p-y^{p-1}u\quad \mbox{for all $i$}. \]
The action of $g=\left(\begin{array}{cc} \alpha&\beta\\ \gamma&\delta
   \end{array}\right)
  \in \GL_2(\bF_p)=\Out(E)$ is given by
  \[g(a)=a^{\alpha}b^{\gamma},
  g(b)=a^{\beta}b^{\delta},\ g(c)=c^{det(g)}.\]
and the action of $g$ on the cohomology is given by 
(\cite{L91}, \cite[p.491]{TY})
\[g^*C=C,\ g^*y_1=\alpha y_1+\beta y_2,
  \ g^*y_2=\gamma y_1+\delta y_2,\ g^*v=(\det(g))v. \]

We consider the $\Out(E)$-module decomposition of $H^*(E)$. 
Let $S^i$ be the homogeneous part of degree $2i$ 
in $\Fp[y_1,y_2]$. Thus, for $0\leq i \leq p-1$, 
\[S^i=\Fp\{y_1^i, y_1^{i-1}y_2, \dots, y_2^i \}.\]  
Recall that $\Fp\{v\}\simeq \det$ as $\Out(E)$-modules 
where $\det$ is the one dimensional determinant representation.  
Then $p(p-1)$ simple $\Fp\Out(E)$-modules 
$$S^iv^q \simeq S^i\otimes (\det)^q\quad  (0\le i\le p-1,0\le q\le p-2)$$
give the complete set of representatives of nonisomorphic  
simple $\Fp\Out(E)$-modules. 
Let us write 
\[ \bCA=\Fp[C,V] \]
and 
\[\bDA=\Fp[D_1,D_2]\]
where $D_1=C^p+V$, $D_2=CV$. Then 
$$\bCA=H^*(E)^{\Out(E)}$$
the $\Out(E)$-invariants, and
$$\bDA|A=H^*(A)^{\Out(A)}$$
for all $A \in \mcl{A}(E)$. Moreover, we have 
$$\bCA=\bDA\{1, C,\dots, C^p\}.$$

Next, we study the $\Out(E)$-module structure of $S^j$ for  $j=(p-1)+i$ with $1\le i\le p-2$.  
Let 
\[T^i=\Fp\{y_1^{p-1}y_2^i,\ y_1^{p-2}y_2^{i+1},\ 
...,\ y_1^iy_2^{p-1}\} \] 
for $0 \leq i \leq p-2$. In particular, $T^0=S^{p-1}$. 
Then we have 
\[S^j=\Fp\{y_1^j,\ y_1^{j-1}y_2,\ ...,\ y_1^p y_2^{i-1},\ y_2^j,\ 
T^i \}\]
since, in $H^*(E)$, we have the relation $y_1^py_2-y_1y_2^p=0$, and hence (for $i\ge 2$)
\[y_1^{j-1}y_2=y_1^{i-1}y_2^p, \
   y_1^{j-2}y_2^2=y_1^{i-2}y_2^{p+1},\ ...,\ y_1^py_2^{i-1}=y_1y_2^{j-1}.\]
Since  
$y_1^j=Cy_1^i,\ ...,\ y_1^{p}y_2^{i-1}=Cy_1y_2^{i-1},\ y_2^j=Cy_2^i$, 
we can write $S^j=CS^i+ T^i$.  In general, we have 
the following. 

\begin{prop}\label{p42}
Let $0 \leq i \leq p-2$ and $m\geq 1$. Then
$$S^{m(p-1)+i}=C^{m-1}(CS^i+T^i)$$
where $S^0=\Fp$ and $T^0=S^{p-1}$. In particular, 
the subalgebra of $H^*(E)$ generated by $y_1$ and $y_2$ is 
written as  
$$\sum_{i=0}^{p-1}S^i+\Fp[C]C(\Fp C+S^{p-1})+
\sum_{i=1}^{p-2}\Fp[C](CS^i+T^i).$$
\end{prop}

The $\Fp$-subspace $CS^i$ is a 
$\GL_2(\Fp)$-submodule of $C^{p-1+i}=CS^i+T^i$. 
Let $\bar{T^i}= S^{p-1+i}/(CS^i)=(CS^i+T^i)/CS^i$ 
for $1 \leq i \leq p-2$. 
 
\begin{lem}[Glover {\cite[(5.7)]{Gl}}]\label{l43} 
As $\GL_2(\Fp)$-modules, 
\[\bar{T}^i\cong (S^{p-1-i}\otimes {\det}^i).\]
\end{lem}

\begin{proof}
Let $g=
\begin{pmatrix}
1 & 1 \\ 0 & 1
\end{pmatrix}$, 
$g'=
\begin{pmatrix}
1 & 0 \\ 1 & 1
\end{pmatrix} \in \GL_2(\Fp)$. 
We set $\overline{y_1^iy_2^{p-1}}=y_1^iy_2^{p-1} + CS^i \in \bar{T^i}$. Then   
$\Fp\{\overline{y_1^iy_2^{p-1}}\}$ is the unique nonzero minimal 
$\Fp\la g \ra$-submodule of $\bar{T^i}$. 
On the other hand, 
$\overline{y_1^iy_2^{p-1}}$ generates $\bar{T^i}$ as an  
$\Fp \la g' \ra$-module, hence it follows that 
$\bar{T^i}$ is a simple $\Fp\GL_2(\Fp)$-module. 
Since $\dim \bar{T^i}=p-i$, it is isomorphic to 
$S^{p-i-1}\ot \det^q$ for some $0 \leq q \leq p-2$. 
Since $\Fp\{y_2^{p-i-1}\ot \det^q\}$ is the unique 
nonzero minimal $\Fp\la g\ra$-submodule of 
$S^{p-i-1}\ot \det^q$, there exists an  
$\Fp\GL_2(\Fp)$-isomorphism $\bar{T^i} \isoarr S^{p-i-1}\ot \det^q$ 
sending $\overline{y_1^iy_2^{p-1}}$ to $y_2^{p-i-1}\ot \det^q$. 
 
Now consider the action of $t=\mathrm{diag}(\lam_1, \lam_2) \in 
\GL_2(\Fp)$. Since 
$$t^*(\overline{y_1^iy_2^{p-1}})=\lam_1^i(\overline{y_1^iy_2^{p-1}})$$
and 
$$t^*(y_2^{p-i-1}\ot {\det}^q)=\lam_2^{p-i-1}(\lam_1\lam_2)^q
(y_2^{p-i-1}\ot {\det}^q),$$
we have $i=q$. This completes the proof of the lemma. 
\end{proof}

From Lemma \ref{l41}, we have the following. 

\begin{thm}\label{t44}
\[ H^*(E)=
\Fp[C,v]\{(\bigoplus_{i=0}^{p-1}S^i)
\oplus(\bigoplus_{i=1}^{p-2}T^i)\}
=
\bCA \{
\bigoplus_{i=0}^{p-2}\bigoplus_{q=0}^{p-2}
(S^iv^q \oplus T^iv^q))\}. \]
\end{thm}

Let $e_{i,q}$ be an idempotent in $\bF_p\GL_2(\Fp)$ corresponding to 
the simple module $S^i \ot \det^q$, namely, 
$(S^i\ot \det^q)e_{i,q}=S^i \ot \det^q$ and $Se_{i,q}=0$ for a simple 
module which is not isomorphic to $S^i \ot \det^q$. 
Let $H_{i,q}=H^*(E)e_{i,q}$. 

\begin{cor}\label{c45} We have the following isomorphisms as $\bF_p$-vector spaces. \\
{\rm(1)} For $0 \leq q \leq p-2$, 
$$ H_{0,q}=\bCA\{v^q\}, \quad 
H_{p-1,q}=\bCA\{S^{p-1}v^q\}.$$ 
{\rm (2)} Let $1\leq i \leq p-2$ and $0 \leq q \leq p-2$. 
Assume $i+q \equiv m \mod p-1$ where $0 \leq m \leq p-2$, 
then  
\[ H_{i, q} \simeq \bCA\{S^iv^q\oplus
            T^{p-1-i}v^m\}. \] 
\end{cor}

We shall need the following lemma in the proof of Proposition \ref{p92}.  

\begin{lem}\label{l46} 
Let $r=n(p-1)-2$ for $0\leq n\leq p-1$. 
If $0 \leq m\leq p-2$, then $vS^{m+1}H^{2r}(E)$ 
has no composition factor isomorphic to $S^m$ as a 
$\GL_2(\bF_p)$-module.
\end{lem}

\begin{proof}
Let $R$ be the subalgebra of $H^*(E)$ generated by 
$y_1,y_2$ and $C$. Then $H^*(E)\simeq R\ot \Fp[v]$. 
Set 
$$K(m)=\bigoplus_{i \equiv m ~\mmod~(p-1)}
(R\{v,\dots,v^{p-2}\}\cap H^{2i}(E)).$$
First, we claim that $vS^{m+1}H^{2r}(E) \subset K(m)$. 
Since $r-(p-2)p <0$, we have 
$$
H^{2r}(E) \subset \bigoplus_{0 \leq i \leq p-3}(R\cap H^{2(r-ip)})v^i 
             \subset \bigoplus_{0 \leq i \leq p-3}R v^i.
$$
Hence $vS^{m+1}H^{2r}(E) \subset R\{v,\dots,v^{p-2}\}$. 
Then we have 
$vS^{m+1}H^{2r}(E) \subset K(m)$ 
since 
$$\deg vS^{m+1}H^{2r}(E)=2(p+m+1+r)\equiv m ~\mmod~ 2(p-1).$$  

Now, we show that  
$K(m)$ has no composition factor isomorphic to 
$S^m$.   
If $m=0$, then the result holds since 
$H_{0,0}=\bCA$ by Corollary \ref{c45} and $\bCA \cap R\{v,\dots,v^{p-2}\}=0$. 
Next assume that $1 \leq m\leq p-2$. Then 
$H_{m,0}=\bCA\{S^m \op T^{p-1-m}v^m \}$ and 
the degree of each nonzero homogeneous part of 
$\bCA\{T^{p-1-m}v^m \}$ is equivalent to $0$ modulo $2(p-1)$. 
On the other hand, since $\bCA\{S^m\} \subset R[V]$, 
we have $\bCA\{S^m\}\cap K(m)=0$. Hence $K(m)$ has 
no simple subquoitent submodule isomorphic to 
$S^m$. This completes the proof. 
\end{proof}

We consider the image of the inflation map from a cyclic 
quotient.      

\begin{lem}\label{l47} 
Let $0 \leq i \leq p-1$. Then the elements
$$y_1^i,(y_1+y_2)^i,\dots,(y_1+iy_2)^i$$
span $S^i$.
\end{lem}
 
\begin{proof}
If $0 \leq k \leq i$, then 
$$(y_1+ky_2)^i=\sum_{j=0}^i
\begin{pmatrix}
i \\ j
\end{pmatrix}
y_1^{i-j}k^jy_2^j=
\sum_{j=0}^ik^j(
\begin{pmatrix}
i \\ j
\end{pmatrix}y_1^{i-j}y_2^j).$$
Since 
$$\begin{pmatrix}
i \\ j
\end{pmatrix}y_1^{i-j}y_2^j,~0 \leq j \leq i$$
is a basis of $S^i$ and the $(i+1)\times (i+1)$-matrix 
$(k^j)_{k,j}$ is invertible, we have that 
$$(y_1+ky_2)^i,~0 \leq k \leq i$$
is a basis of $S^i$.
\end{proof}

\begin{lem}\label{l48} 
Let $p \leq n$ and $\lam_1 ,\lam_2 \in \Fp$. 
Then 
$$(\lam_1y_1+\lam_2y_2)^n=C(\lam_1 y_1+\lam_2y_2)^{n-p+1}.$$
\end{lem}

\begin{proof}
Since 
$$(\lam_1y_1+\lam_2y_2)^p=\lam_1y_1^p+\lam_2y_2^p=
C(\lam_1y_1+\lam_2y_2)$$
we have 
$$(\lam_1y_1+\lam_2y_2)^n=(\lam_1y_1+\lam_2y_2)^p
(\lam_1y_1+\lam_2y_2)^{n-p}=
C(\lam_1y_1+\lam_2y_2)^{n-p+1}.$$
\end{proof}

\begin{cor}\label{c49}
Let $Q\leq  E$ be a subgroup of order $p$. Then  
$$\sum_{\vphi:E \thrarr Q} \vphi^*H^+(Q)=\Fp[C]
(\bigoplus_{1\leq i \leq p-1} S^i).$$ 
\end{cor}

\begin{proof}
Since 
$$\sum_{\vphi:E \thrarr Q} \vphi^*(H^{2n}(Q))=
\sum_{\lam_1, \lam_2 \in \bF_p}(\lam_1 y_1+ \lam_2 y_2)^n$$
for $n>0$, the results follows from Lemma \ref{l47} and Lemma 
\ref{l48}. 
\end{proof}


\section{Cohomology of $A$}\label{HA}
In this section, we study the structure of $H^*(A)$, where 
$A$ is a rank $2$ elementary abelian $p$-subgroup of $E$ so that 
$H^*(A)=\Fp[y,u]$.  
We fix a quotient map
         \[ q : E\lorarr E/\la c\ra \isoarr A\]
and consider the induced map $q^*:H^*(A) \lorarr  H^*(E)$.  
We may assume that $q^*(y)=y_1$ and $q^*(u)=y_2$. 
This map is a $\GL_2(\bF_p)$-module morphism. Here, we 
identify $\Out(E)$, $\Out(A)$ and $\GL_2(\bF_p)$. 
The kernel of $q^*$ is given by
  \[ \Ker q^*=H^*(A)d_2 = \Fp[y,u]d_2\]
where $d_2=yu^p-y^pu$.  
The image of $q^*$ is a subalgebra of 
$H^*(E)$ generated by $y_1$ and $y_2$.

Let $S(A)^i=H^{2i}(A)$. Then $q^*$ induces a $\GL_2(\bF_p)$-isomorphism $S(A)^i \simeq S^i$ for $0 \leq i \leq p-1$. 
Moreover $S(A)^i\ot \det^q$, $0\leq i \leq p-2$, 
$0 \leq q \leq p-2$ give the complete set of representatives 
of isomorphism classes of simple $\GL_2(\bF_p)$-modules.  

Next we described simple $A_p(A,A)$-modules.  
Let $Q\leq A$ be a subgroup of order $p$. Then $\Out(Q)$ is 
a cyclic group of order $p-1$. 
Let $U_i=\Fp u_i ~(0 \leq i \leq p-2)$ be the simple right 
$\Out(Q)$-module defined by 
$u_i \sigma =\lam^i u_i$ where $\Out(Q)=\la \sigma \ra$ and  
$\bF_p^{\times}=\la \lam \ra$. 
Note that if $n\equiv i \pmod {p-1}$ where $0 \leq i \leq p-1$, 
then  
$$H^{2n}(Q) \simeq U_i$$
as right $kQ$-modules.  Since $\Tr_Q^A(H^*(Q))=0$, we have 
$$H^*(Q)A_p(A,Q)=\sum_{\vphi:A \thrarr Q}\vphi^*(H^*(Q))=
\sum_{n \geq 0}\sum_{\lam_1, \lam_2 \in \bF_p}\Fp(\lam_1y+\lam_2 u)^n.$$
In particular, if $1 \leq n \leq p-1$, 
we have $H^{2n}(Q)A_p(A,Q)=H^{2n}(A)$. 
Since 
$$\bigcap_{R<A} \Ker (\res^A_R:H^*(A) \lorarr H^*(R))=
d_2H^*(A),$$ 
it follows that 
$$\bigcap_{\phi \in A(Q,A)}\Ker (\phi:H^n(A) \lorarr H^n(Q))=0$$
for $1 \leq n \leq p-1$. 
Hence $H^{2n}(A) \simeq S(A,Q,U_n)$ for $1 \leq n \leq p-2$ 
and $H^{2(p-1)}(A) \simeq S(A,Q,U_0)$ by Proposition \ref{l31}, and these modules give all simple $A_p(A,A)$-module with minimal subgroup $Q$ by Proposition \ref{l32}.  
Hence we have the following classification of simple 
$A_p(A,A)$-modules. 

\begin{prop}[Harris and Kuhn {\cite[Example 8.1.]{HK}}]\label{p51}
Up to isomorphisms, the simple $A_p(A,A)$-modules are given as follows.
\\
{\rm(1)} $S(A,A,S(A)^i \ot \det^q)$ $(0 \leq i\leq p-1,~ 0 \leq q \leq p-2)$. \\
{\rm(2)} $S(A,Q,U_i)$ $(0 \leq i \leq p-2)$,  
$$\dim S(A,Q,U_i)=
\left\{
\begin{array}{cc}
p & (i=0) \\
i+1  & (1 \leq 0 \leq p-2).
\end{array}\right.$$   
{\rm(3)} One dimensional module with trivial minimal subgroup. 
\end{prop}

Now we consider the structure of $H^*(A)$ as an 
$A_p(A,A)$-module. First, we see $\GL_2(\bF_p)$-module structure 
more closely and specify all simple submodules of $H^*(A)$ isomorphic to $S^{p-1}\ot \det^i$ for $0\leq i \leq p-2$. 

Since $C^jS^{p-1}$ is contained in $\Img q^*$ for any $j \geq 0$, 
we have the following short exact sequence of 
$\GL_2(\bF_p)$-modules, 
$$0 \lorarr d_2H^*(A) \lorarr (q^*)^{-1}(C^jS^{p-1})
\lorarr C^jS^{p-1} \lorarr 0.$$
Since $C^jS^{p-1}$ is projective as an  
$\Fp\GL_2(\bF_p)$-module, 
this exact sequence splits. Hence, in particular, 
for $0 \leq j \leq p-1$, 
there exist $\GL_2(\bF_p)$-submodules $W_j$ such that 
$$(q^*)^{-1}(C^jS^{p-1})=W_j \op d_2H^*(A).$$
Note that since
$$y_1^{(j+1)(p-1)},~y_1^{(j+1)(p-1)-1}y_2,\dots, 
y_1^{(j+1)(p-1)-(p-2)}y_2^{p-2},~y_2^{(j+1)(p-1)}$$
is a basis of $C^jS^{p-1}$, it follows that  
$$q^*(y^{(j+1)(p-1)}),~ q^*(y^{(j+1)(p-1)-1}u), \dots, 
q^*(y^{(j+1)(p-1)-(p-2)}u^{p-2}),~ q^*(u^{(j+1)(p-1)})$$
is a basis of $C^jS^{p-1}$. 

Next, we consider Dickson subalgebra of $H^*(A)$. 
Let $\Tilde{D_1}=\res^E_A(D_1)=y^{p(p-1)}+\res^E_A(V)$ and $\Tilde{D_2}=\res^E_A(D_2)=d_2^{p-1}$. 
Let $\Tilde{\bDA} =\Fp[\Tilde{D_1},\Tilde{D_2}]$. Then 
$\Tilde{\bDA}=H^*(A)^{\Out(A)}$ and the restriction map induces an 
isomorphism 
$$\res^E_A: \bDA \isoarr \Tilde{\bDA}.$$

\begin{lem}\label{l52}
We have 
$$\Tilde{D_1}y^{(j+1)(p-1)-l}u^l \equiv y^{(p+j+1)(p-1)-l}u^l 
\pmod {d_2H^*(A)}$$
for any $0 \leq l \leq p-2$ and 
$$\Tilde{D_1}u^{(j+1)(p-1)} \equiv u^{(p+j+1)(p-1)} 
\pmod {d_2H^*(A)}.$$
In particular, 
$$(q^*)^{-1}(C^{j+mp}S^{p-1})=\Tilde{D_1}^m W_j\op 
d_2H^*(A)$$
for any $m \geq 0$ and $0 \leq j \leq p-1$.
\end{lem}

\begin{proof}
First, we have 
\begin{eqnarray*}
\Tilde{D_1}y^{(j+1)(p-1)-l}u^l &=& 
(y^{p(p-1)}+\res^E_A(v)^{p-1})y^{(j+1)(p-1)-l}u^l\\ 
&\equiv& y^{(p+j+1)(p-1)-l}u^l
\end{eqnarray*}
for $0 \leq l \leq p-2$ since 
$\res^E_A(v)y=d_2$. On the other hand, modulo $d_2H^*(A)$,  
$$y^{p(p-1)}u^{(j+1)(p-1)}\equiv 
(y^pu)^{p-1}u^{j(p-1)} \equiv 
(yu^p)^{p-1}u^{j(p-1)}=y^{p-1}u^{(p+j)(p-1)}$$
$$(u^p-y^{p-1}u)^{p-1} \equiv (u^p-y^{p-1}u)u^{p(p-2)}$$
and 
\begin{eqnarray*}
(u^p-y^{p-1}u)^{p-1}u^{(j+1)(p-1)} &\equiv& 
(u^p-y^{p-1}u)u^{(p+j)(p-1)-1} \\ 
&=&
u^{(p+j+1)(p-1)}-y^{p-1}u^{(p+j)(p-1)}.
\end{eqnarray*}
Hence we have
$$\Tilde{D}_1u^{(j+1)(p-1)}=
(y^{p(p-1)}+(u^p-y^{p-1}u)^{p-1})u^{(j+1)(p-1)}
\equiv u^{(p+j+1)(p-1)}. $$
\end{proof}

There exists a sequence of $\GL_2(\bF_p)$-submodules, 
$$H^*(A) \supset d_2H^*(A) \supset d_2^2H^*(A) \supset 
\cdots.$$
We shall consider each factor 
module $d_2^mH^*(A)/d_2^{m+1}H^*(A)$. 
Note that, since 
$d_2g=g^*(d_2)=(\det g)d_2$ for any $g\in \GL_2(\bF_p)$ and 
$\Tilde{\bDA}=H^*(A)^{\Out(A)}$,  
it follows that 
$$\Tilde{D_1}^i\Tilde{D_2}^j d_2^mW_n\simeq S(A)^{p-1}\ot {\det}^m$$
for any $m$. 
First, we consider the factor module $H^*(A)/d_2H^*(A)$. 

\begin{lem}\label{l53}
There exists a sequence of $\GL_2(\bF_p)$-submodules, 
$$H^*(A) \supset (\Fp[\Tilde{D_1}]\{\bigoplus_{n=0}^{p-1}W_n\} 
\op d_2H^*(A)) \supset d_2H^*(A).$$
Moreover, 
$$H^*(A)/(\Fp[\Tilde{D_1}]\{\bigoplus_{n=0}^{p-1}W_n\} 
\op d_2H^*(A))$$
has no simple subquotient module which is isomorphic to 
$S(A)^{p-1}\ot \det^i$ for any $i$.  
\end{lem}

\begin{proof}
By Lemma \ref{l52}, $q^*$ induces isomorphisms, 
$$H^*(A)/(\Fp[\Tilde{D_1}]\{\bigoplus_{n=0}^{p-1}W_n\} 
\op d_2H^*(A)) \simeq \Fp[y_1,y_2]/\Fp[C]S^{p-1}$$
and 
$$\Fp[\Tilde{D_1}]\{\bigoplus_{n=0}^{p-1}W_n\} \simeq  (\Fp[\Tilde{D_1}]\{\bigoplus_{n=0}^{p-1}W_n\} 
\op d_2H^*(A))/d_2H^*(A) \simeq \Fp[C]S^{p-1}.$$
Hence the result follows from Proposition \ref{p42}. 
\end{proof}

Next we consider the factor module 
$d_2^mH^*(A)/d_2^{m+1}H^*(A)$ for general $m \geq 0$. 
Since $d_2^m$ induces an isomorphism, 
$$(H^*(A)/d_2H^*(A)) \ot {\det}^m \simeq 
d_2^mH^*(A)/d_2^{m+1}H^*(A),$$
we have the following. 

\begin{lem}\label{l54}
There exists a sequence of $\GL_2(\bF_p)$-modules, 
$$d_2^m H^*(A) \supset (\Fp[\Tilde{D_1}]\{\bigoplus_{n=0}^{p-1}
d_2^m W_n\} 
\op d_2^{m+1}H^*(A)) \supset d_2^{m+1}H^*(A).$$
Moreover, 
$$d_2^m H^*(A)/(\Fp[\Tilde{D_1}]\{\bigoplus_{n=0}^{p-1}d_2^m W_n\} 
\op d_2^{m+1}H^*(A))$$
has no simple subquotient module which is isomorphic to 
$S(A)^{p-1}\ot \det^i$ for any $i$.  
\end{lem}

Since $d_2^{p-1}=\Tilde{D_2}$ and $\Tilde{\bDA}=\Fp[\Tilde{D_1},\Tilde{D_2}]$, 
we have the following by Lemma \ref{l53} and \ref{l54}. 

\begin{prop}\label{p55}
For any $0 \leq m \leq p-2$, the submodule 
$$\Tilde{\bDA}\{\bigoplus_{n=0}^{p-1}d_2^m W_n\}$$ 
is a sum of all simple submodules of $H^*(A)$ isomorphic to $S(A)^{p-1}\ot {\det}^m$.  
In particular, 
$$H^*(A)/\Tilde{\bDA}\{\bigoplus_{n=0}^{p-1}
d_2^mW_n\}$$
has no simple subquotient module isomorphic to 
$S(A)^{p-1}\ot {\det}^m$. 
\end{prop}

Now, we consider $A_p(A,A)$-module structure of $H^*(A)$. 
Let $J(A)$ be the ideal of $A_p(A,A)$ generated by 
bisets which factor through a proper subgroup of $A$.
If $L^*=\op_{n \geq 0}L^n$ is a graded vector subspace of 
$H^*(E)$, we set 
$L^+=\op_{n>0}L^n$, the positive degree part of $L^*$. 

\begin{prop}\label{p56} 
Let $Q \leq A$ be a subgroup of order $p$. 
Let 
$$L^n(Q)=H^n(Q)A_p(A,Q)+(d_2H^*(A)\cap H^n(A))$$ 
and 
$$L^*(Q)= \bigoplus_{n \geq 0}L^n(Q)=H^*(Q)A_p(A,Q)+d_2H^*(A).$$ 
Then we have the following. \\
{\rm(1)} We have a sequence of $A_p(A,A)$-submodules, 
$$H^*(A)\supset L^*(Q) \supset 
d_2H^*(A).$$
{\rm(2)}  We have 
$H^*(A)J(A) \subset H^*(Q)A_p(A,Q)$. In particular, 
every simple subquotient module of $H^*(A)/L^*(Q)$
has minimal subgroup $A$. Moreover $H^*(A)/L^*(Q)$
has no simple subquotient module isomorphic to 
$S(A,A,S(A)^{p-1}\ot \det^i)$ for any $i$. 
\\
{\rm(3)} The factor module  
$L^+(Q)/d_2H^*(A)$ is a direct sum of 
$A_p(A,A)$-modules with minimal subgroup $Q$. 
More precisely, if $n>0$ and $n \equiv i \pmod {p-1}$ where 
$0 \leq i \leq p-2$,  then 
$$L^{2n}(Q)/(d_2H^*(A)\cap H^{2n}(A)) \simeq S(A,Q,U_i).$$ 
\\
{\rm(4)} We have $d_2H^*(A)J(A)=0$. In particular, 
every simple subquotinet module of $d_2H^*(A)$ 
has minimal subgroup $A$. 
\end{prop}

\begin{proof} We have $d_2H^*(A)J(A)=0$ 
since the restriction 
of $d_2$ to any proper subgroup is zero. 
Hence 
$d_2H^*(A)$ is an $A_p(A,A)$-submodule  
and the claims in (1) and (4) are proved. 

Since $A$ is an elementary abelian $p$-group, 
any transfer from proper subgroup to $A$ is zero. Hence 
it follows that 
$$H^*(A)J(A) \subset 
H^*(Q)A_p(A,Q)=\sum_{\vphi:A \thrarr Q}\vphi^*H^*(Q)=
\sum_{n \geq 0}
\sum_{\lam_1, \lam_2 \in \bF_p}(\lam_1y+\lam_2u)^n.$$
Then  
$q^*$ induces an isomorphism of $\mathrm{GL}_2(\bF_p)$-modules, 
$$H^*(A)/L^*(Q) \simeq 
\Fp[y_1,y_2]/\Fp[C]\{ \sum_{i=1}^{p-1}S^i\}$$
and 
$$L^*(Q)/d_2H^*(A) \simeq \Fp[C]\{\sum_{i=1}^{p-1}S^i\}$$
by Corollary \ref{c49},  
where $\Fp[y_1,y_2]$ is a subalgebra of $H^*(E)$ generated 
by $y_1$ and $y_2$. 
It follows that no simple subquotient module of $H^*(A)/L^*(Q)$ 
is isomorphic to $S(A,A,S(A)^{p-1}\ot {\det}^i)$ by Proposition  
\ref{p42}.

On the other hand, since
$\bigcap_{R<A} \Ker \res^A_R=d_2H^*(A)$, 
it follows that 
$$\bigcap_{\phi \in A(Q,A)}\Ker (\phi:H^*(Q)A_p(A,Q) \lorarr 
H^*(Q))=H^*(Q)A_p(A,Q)\cap d_2H^*(A).$$
Since $H^+(Q)$ is a direct sum of simple $A(Q,Q)$-modules with 
minimal subgroup $Q$, it follows that 
$$L^+(Q)/d_2H^*(A) \simeq 
H^+(Q)A_p(A,Q)/(H^+(Q)A_p(A,Q)\cap d_2H^*(A))$$
is a direct sum of simple $A_p(A,A)$-modules with minimal 
subgroup $Q$ by Lemma \ref{l32}. This completes the proof of 
(2) and (3). 
\end{proof}

Finally, we shall specify simple $A_p(A,A)$-submodules of $H^*(A)$  isomorphic to $S(A,A,S(A)^{p-1}\ot \det^i)$. 

\begin{prop}\label{p57} Let 
$i,j \geq0$, $0 \leq m \leq p-2$ and 
$0 \leq n \leq p-1$. If $j+m>0$, then   
${\Tilde{D_1}}^i\Tilde{D_2}^jd_2^mW_n$ is an $A_p(A,A)$-submodule 
with minimal subgroup $A$, namely, 
$${\Tilde{D_1}}^i\Tilde{D_2}^jd_2^mW_n \simeq 
S(A,A,S(A)^{p-1}\ot {\det}^m)$$
as $A_p(A,A)$-modules. The quotient module 
$$H^*(A)/(\Tilde{\bDA}\{\bigoplus_{n=0}^{p-1}\Tilde{D_2}W_n\})$$
has no simple subquotient module which is isomorphic to $S(A,A,S(A)^{p-1})$. On the other hand, if $m>0$, then    
the quotient module 
$$H^*(A)/(\Tilde{\bDA}\{\bigoplus_{n=0}^{p-1}d_2^mW_n\})$$
has no simple subquotient module which is isomorphic to $S(A,A,S(A)^{p-1}\ot \det^m)$.
\end{prop}

\begin{proof} First, note that  
${\Tilde{D_1}}^i\Tilde{D_2}^jd_2^mW_n \simeq 
S(A)^{p-1}\ot {\det}^m$ as $\mathrm{GL}_2(\bF_p)$-modules. 
By Lemma \ref{p56}(4),   
${\Tilde{D_1}}^i\Tilde{D_2}^jd_2^mW_n$ is an $A_p(A,A)$-submodule 
and  $${\Tilde{D_1}}^i\Tilde{D_2}^jd_2^mW_n \simeq 
S(A,A,S(A)^{p-1}\ot {\det}^m). $$  
If $1 \leq m \leq p-2$, then, as a 
$\mathrm{GL}_2(\bF_p)$-module,  
$$H^*(A)/(\Tilde{\bDA}\{\bigoplus_{n=0}^{p-1}d_2^mW_n\})$$
has no simple subquotient module which is isomorphic to 
$S(A)^{p-1}\ot {\det}^m$ by Proposition \ref{p55}. Hence,  
as an $A_p(A,A)$-module, 
$$H^*(A)/(\Tilde{\bDA}\{\bigoplus_{n=0}^{p-1}d_2^mW_n\})$$
has no simple subquotient module which is isomorphic to $S(A,A,S(A)^{p-1}\ot \det^m)$.

Next suppose $m=0$. By Lemma \ref{p56} (2)(3), 
$H^*(A)/d_2H^*(A)$ has no subquotient module which 
isomorphic to $S(A,A,S(A)^{p-1})$. 
Hence it follows that  
$\Tilde{\bDA}\{\op_{n=0}^{p-1}W_n\}/\Tilde{\bDA}
\{\op_{n=0}^{p-1}\Tilde{D_2}W_n\}$ 
has no subquotient module which is isomorphic to 
$S(A,A,S(A)^{p-1})$
since 
$$\Tilde{\bDA}\{\bigoplus_{n=0}^{p-1}W_n\} \cap 
d_2H^*(A)=\Tilde{\bDA}\{\bigoplus_{n=0}^{p-1}\Tilde{D_2}W_n\}.$$
On the other hand, 
$H^*(A)/\Tilde{\bDA}\{\op_{n=0}^{p-1}W_n\}$ has no 
subquotient module 
which is isomorphic to $S(A,A,S(A)^{p-1})$ by Proposition \ref{p55}. 
This completes the proof of the proposition. 
\end{proof}


\section{Images of transfer maps}\label{transfer}

We study the image of transfer from maximal elementary 
abelian $p$-subgroups. 
Let $M_i=CS^i+T^i$ for $0\leq i\leq p-2$ where 
$S^0=\Fp$, $T^0=S^{p-1}$. 
Moreover, let $\hat{y}_A=y_1$ if $A=A_i$, $i \in \bF_p$ and 
$\hat{y}_A=y_2$ if $A=A_{\infty}$. Note that 
$\res^E_A(\hat{y}_A)=y$ for any $A \in \mcl{A}(E)$.  

\begin{lem}\label{l61}
Let $A \in \mcl{A}(E)$.  \\
{\rm(1)} 
We have 
\[
\Tr_{A}^E(u^{p-1})=
\left\{
\begin{array}{cc}
(iy_1-y_2)^{p-1}-C & (A=A_i,~0\leq i\leq p-1)\\
y_1^{p-1}-C        & (A=A_{\infty}).
\end{array}\right.
\]
{\rm(2)} 
Suppose that $0\leq j \leq p-2$ and $0\leq m \leq p$, then 
$$\Tr_A^E(u^{m(p-1)+j})=
\left\{\begin{array}{cl}
0 & (0 \leq m \leq j \leq p-2) \\
\begin{pmatrix}
m-1 \\ j 
\end{pmatrix}
v^jC^{m-j-1}\Tr_A^E (u^{p-1}) &  (0 \leq j<m,~j \leq p-2).
\end{array}\right.$$
In particular, 
$$\Tr_A^E(u^{m(p-1)+j}) \in \Fp[C]\{v^jM_0\}.$$
Moreover, for any $l \geq 1$, 
$$C^{l-1}\Tr_A^E (u^{p-1})=(-1)^{l-1}(\Tr(u^{p-1}))^l.$$
{\rm(3)} If $0 \leq m \leq p-2$ then 
$\hat{y}_A^m\Tr_A^E(u^{p-1})~(A\in\mcl{A}(E))$ is a basis of $M_m$. 
\\
{\rm(4)} For $1 \leq n \leq p$, 
\begin{eqnarray*}
\Tr_A^E ((y^{p-1}-u^{p-1})^n)&=&-(C+\Tr(u^{p-1}))^n+C^n \\
&=& \left\{\begin{array}{cl}
-(iy_1-y_2)^{n(p-1)}+C^n & (A=A_i, ~0\leq i \leq p-1)\\
-y_1^{n(p-1)}+C^n & (A=A_{\infty}).
\end{array}\right.
\end{eqnarray*}
In particular, 
$\Tr_A^E ((y^{p-1}-u^{p-1})^n) \in \Fp[C]\{M_0\}$. 
\end{lem}

\begin{proof} 
Since $\res^E_A (v)=u^p-y^{p-1}u$ and $\res^E_A(C)=y^{p-1}$, 
we have $u^p=\res^E_A(v)+(\res^E_A(C))u$. \\
(1) First, note that by the double coset formula,  
\[
\res^E_{A_i}\Tr_{A_j}^E(u^{p-1})=
\left\{
\begin{array}{cl}
-y^{p-1} & (i=j) \\
0        & (i \ne j)
\end{array}\right. \tag{\ref{l61}.1}\]
for $i,j =0,1,\dots,p-1, \infty$. 
Let $0\leq i,j \leq p-1$. Then 
\begin{eqnarray*}
\res^E_{A_i}((jy_1-y_2)^{p-1}-C)&=&(jy-iy)^{p-1}-y^{p-1}=
((j-i)^{p-1}-1)y^{p-1}\\
&=&
\left\{
\begin{array}{cl}
-y^{p-1} & (i=j) \\
0        & (i\ne j)
\end{array}\right. 
\end{eqnarray*}
and 
$$\res^E_{A_{\infty}}((jy_1-y_2)^{p-1}-C)=(-y)^{p-1}-y^{p-1}=0.$$
On the other hand, 
$$\res^E_{A_i}(y_1^{p-1}-C)=
\left\{
\begin{array}{cc}
y^{p-1}-y^{p-1}=0 & (0 \leq i \leq p-1) \\
-y^{p-1}          & (i = \infty). 
\end{array}\right. $$
By Quillen's theorem \cite{Q}, we get the equation. \\
(2) First assume that $0 \leq m \leq j \leq p-2$. 
We proceed by induction on $m$. The case $m=0$ is trivial. 
Let $0 < m\leq j \leq p-2$. 
Then 
\begin{eqnarray*}
u^{m(p-1)+j}&=&u^pu^{(m-1)(p-1)+(j-1)} \\
&=& (\res^E_A(v)+\res^E_A(C)u)u^{(m-1)(p-1)+(j-1)}\\
&=& \res^E_A(v)u^{(m-1)(p-1)+(j-1)}+\res^E_A(C)u^{(m-1)(p-1)+j}.
\end{eqnarray*}
Hence 
$$\Tr_A^E (u^{m(p-1)+j})=v\Tr_A^E(u^{(m-1)(p-1)+(j-1)})
+C\Tr_A^E(u^{(m-1)(p-1)+j})=0$$
by induction. Next assume $0\leq j<m \leq p$, $j\leq p-2$. 
If $m=1$ then $j=0$ and the result is trivial. 
Assume $1<m \leq p$. If $j=0$ then, 
\begin{eqnarray*}
u^{m(p-1)}&=&u^pu^{(m-2)(p-1)+(p-2)} \\
&=& (\res^E_A(v)+\res^E_A(C)u)u^{(m-2)(p-1)+(p-2)}\\
&=& \res^E_A(v)u^{(m-2)(p-1)+(p-2)}+\res^E_A(C)u^{(m-1)(p-1)}.
\end{eqnarray*}
Hence 
$$\Tr_A^E (u^{m(p-1)})=C\Tr_A^E(u^{(m-1)(p-1)})=C^{m-1}\Tr_A^E(u^{p-1}).$$
Note that $m-2 \leq p-2$ since $m\leq p$. 
On the other hand, if $0<j<m\leq p$, $j \leq p-2$, then 
\begin{eqnarray*}
u^{m(p-1)+j}&=&u^pu^{(m-1)(p-1)+(j-1)} \\
&=& (\res^E_A(v)+\res^E_A(C)u)u^{(m-1)(p-1)+(j-1)}\\
&=& \res^E_A(v)u^{(m-1)(p-1)+(j-1)}+\res_A^E(C)
u^{(m-1)(p-1)+j}.
\end{eqnarray*}
Hence, if $j<m-1$, then 
\begin{eqnarray*} &&
\Tr_A^E (u^{m(p-1)+j}) \\&=&
v
\begin{pmatrix}
m-2 \\ j-1
\end{pmatrix}
v^{j-1}C^{m-j-1}\Tr_A^E(u^{p-1})+C 
\begin{pmatrix}
m-2 \\ j
\end{pmatrix}
v^j C^{m-j-2}\Tr_A^E(u^{p-1})\\
&=& 
\begin{pmatrix}
m-1 \\ j
\end{pmatrix}
v^jC^{m-j-1}\Tr_A^E(u^{p-1}).
\end{eqnarray*}
If $j=m-1$, then  
$$\Tr _A^E(u^{m(p-1)+j})=
v
\begin{pmatrix}
m-2 \\ j-1
\end{pmatrix}
v^{j-1}C^{m-j-1}\Tr_A^E(u^{p-1})
=v^jC^{m-j-1}\Tr_A^E(u^{p-1}).$$
On the other hand, since 
$$\res^E_{A'}(\Tr_A^E(u^{p-1}))=
\left\{\begin{array}{cl}
-y^{p-1}&(A'=A) \\
0       &(A' \ne A),  
\end{array}\right.$$
we have 
$$\res^E_{A'}(\Tr_A^E(u^{p-1}))^l=\res^E_{A'}((-1)^{l-1}C^{l-1}\Tr_A^E(u^{p-1}))$$ 
for any $A'$. 
\\
(3) 
By (1), $\hat{y}_A^m\Tr_A^E(u^{p-1}) \in M_m$. On the 
other hand, 
$\hat{y}_A^m\Tr_A^E(u^{p-1})$, $A\in\mcl{A}(E)$ 
are linearly independent by (\ref{l61}.1). Hence the results follows 
since $\dim M_m=p+1$.
\\
(4) By (2), we have 
\begin{eqnarray*}
\Tr_A^E ((y^{p-1}-u^{p-1})^n)&=&
\Tr_A^E(\sum_{j=0}^n 
\begin{pmatrix}n \\ j \end{pmatrix}
(-1)^jy^{(p-1)(n-j)}u^{(p-1)j}) \\
&=& \sum_{j=1}^n\begin{pmatrix}n \\ j \end{pmatrix}
(-1)^jC^{n-j}(-1)^{j-1}(\Tr_A^E(u^{p-1}))^j \\
&=& -\sum_{j=0}^n
\begin{pmatrix}n \\ j \end{pmatrix}
C^{n-j}(\Tr_A^E(u^{p-1}))^j+C^n \\
&=&-(C+\Tr_A^E(u^{p-1}))^n+C^n.
\end{eqnarray*}
\end{proof}

\begin{lem}\label{l62} 
{\rm (1)} If $m \geq 0$ and $1 \leq i\leq p-2$, then 
$$\Tr_A^E(u^{m(p-1)+i}) \in \bCA \{v^iM_0\}.$$
{\rm (2)} If $m \geq 1$, then 
$$\Tr_A^E(u^{m(p-1)}) \equiv C^{m-1}\Tr_A^E(u^{p-1}) 
\mod \bCA \{VM_0\}.$$
In particular, 
$$\Tr_A^E(u^{m(p-1)}) \in \bCA \{M_0\}.$$
{\rm(3)} For any $n \geq 0$, we have 
$$\Tr_A^E((u^{p-1}-y^{p-1})^n) \in \bCA \{M_0\}.$$
\end{lem}

\begin{proof} 
Since $\res^E_A(v)=u^p-y^{p-1}u$, we have 
$$\Tr_A^E(u^k)=C\Tr_A^E(u^{k-(p-1)})
+v\Tr_A^E(u^{k-p})$$
for $k \geq p$. 

We prove (1) and (2) by induction on $m$. If $m=0$ then 
$\Tr_A^E(u^i)=0$. If $m=1$ then the results follows by 
Lemma \ref{l61}. Assume that $m>1$. If 
$1 \leq i \leq p-2$, then 
$$\Tr_A^E(u^{m(p-1)+i})=
C\Tr_A^E(u^{(m-1)(p-1)+i})
+v\Tr(u^{(m-1)(p-1)+i-1})$$
and the right hand side is contained in $\bCA \{v^iM_0\}$ by  
induction.  

On the other hand, modulo $\bCA \{VM_0\}$,  
\begin{eqnarray*}
\Tr_A^E(u^{m(p-1)})&=&C\Tr_A^E(u^{(m-1)(p-1)})
+v\Tr(u^{(m-2)(p-1)+(p-2)}) \\
& \equiv &
C\Tr_A^E(u^{(m-1)(p-1)}) \\
& \equiv &
CC^{m-2}\Tr_A^E(u^{p-1}) \\
&= &
C^{m-1}\Tr_A^E(u^{p-1})
\end{eqnarray*}
by induction. Since 
$\Tr_A^E(u^{m(p-1)}y^{k(p-1)})=C^k\Tr_A^E(u^{m(p-1)})$, 
(3) follows from (2). 

\end{proof}

\begin{lem}\label{l63}
Suppose $1 \leq q \leq p-2$. 
Let $\vphi : A \lorarr A' \leq E$ be an isomorphism. Then 
$$\Tr_A^E \vphi^*(\res^E_{A'}(v)^q) \in C^{q-1}M_q.$$
\end{lem}

\begin{proof}
For any $\vphi$,  $\vphi^*\res^E_{A'}(v)=\lam\mu$ where 
$\lam \in \bF_p^{\times}$ and $\mu$ equals to $\res^E_A(v)$ 
or some $E$-conjugate of $y(u^{p-1}-y^{p-1})$. If 
$\vphi^*\res^E_{A'}(v)=\lam\res^E_A(v)$ then 
$\Tr_A^E\vphi^*\res^E_{A'}(v)=0$. On the other hand, 
if $\vphi^*\res^E_{A'}(v)=\lam\mu$ and $\mu$ is $E$-conjugate 
of $y(u^{p-1}-y^{p-1})$, then 
$\Tr_A^E\vphi^*(\res^E_{A'}(v)^q) \in \Fp 
\Tr_A^E(y^q(u^{p-1}-y^{p-1})^q) 
\subset \hat{y}_A^q C^{q-1}M_0 \subset C^{q-1}M_q$ 
by Lemma \ref{l61}(4).
\end{proof}

\begin{lem}\label{l64}
Let $\vphi : A \lorarr A' \leq E$ be an isomorphism.
Then
$$\Tr_A^E\vphi^*(\res^E_{A'}(v^j S^j))=0$$
for any $j \geq 0$. 
\end{lem}

\begin{proof}
Since $\vphi^*(\res^E_{A'}(v^jS^j)) \subset 
\Fp d_2^j$, the image of transfer map is zero. 
\end{proof}

For $A' \in \mcl{A}(E)$, let $y' \in H^2(A')$ be the element corresponding to $y \in H^2(A)$, that is,  
$y'=\res^E_{A'}(\hat{y}_{A'})$. Note that 
$\res^E_{A'}(S^i)=\Fp (y')^i$ and 
$\res^E_{A'}(C)=(y')^{p-1}$. 

\begin{lem}\label{l65}
Suppose that $0 \leq m\leq p$ and $0 \leq i , q \leq p-2$. 
Let $\vphi : A \lorarr A' \leq E$ be 
an isomorphism. 
Then the element 
$$\Tr_A^E\vphi^*(\res^E_{A'}(C^mv^q)(y')^i)$$
is contained in the following $\Fp$-subspace $W$:
$$\begin{array}{lc}
\mbox{{\rm \hspace{.7cm} condition}} & W \\ \hline 
q \leq i,~m+q<i+1 & 0 \\
q \leq i,~m+q \geq i+1& C^{m+q-i-1}v^iM_q \\
i<q,~m>0,~m+q<p+i+1 & 0 \\
i<q,~m>0,~m+q\geq p+i+1 & C^{m+q-p-i-1}Vv^iM_q \\
i<q,~m=0 & C^{q-i-1}v^iM_q
\end{array}$$
\end{lem}

\begin{proof}
Assume $q \leq i$. Then 
\begin{eqnarray*}
\Tr_A^E\vphi^*(\res^E_{A'}(C^mv^q)(y')^i) 
&=&  
\Tr_A^E\vphi^*((\res^E_{A'}(v)y')^q \res^E_{A'}(C^m)(y')^{i-q}) \\
&\subset& 
\Fp(v\hat{y}_A)^q\Tr_A^E\vphi^*((y')^{m(p-1)+(i-q)}) \\ 
&\subset& 
\Fp(v\hat{y}_A)^q\Tr_A^E(u^{m(p-1)+(i-q)})
\end{eqnarray*}
and the result follows from Lemma \ref{l61}. Next, 
assume $i<q$ and $m>0$. Then 
\begin{eqnarray*}
\Tr_A^E\vphi^*(\res^E_{A'}(C^mv^q)(y')^i) 
&=&  
\Tr_A^E\vphi^*((\res^E_{A'}(v)y')^q \res^E_{A'}(C^{m-1})
(y')^{p-1-(q-i)}) \\
&\subset& 
\Fp(v\hat{y}_A)^q\Tr_A^E\vphi^*((y')^{(m-1)(p-1)+(p-1-(q-i))}) \\ 
&\subset& 
\Fp(v\hat{y}_A)^q\Tr_A^E(u^{(m-1)(p-1)+(p-1-(q-i))})
\end{eqnarray*}
and the result follows from Lemma \ref{l61}. 
Finally, assume $i<q$ and $m=0$. Then 
\begin{eqnarray*}
\Tr_A^E\vphi^*(\res^E_{A'}(v^q)(y')^i) 
&=&  
\Tr_A^E\vphi^*((\res^E_{A'}(v)y')^i \res^E_{A'}(v)^{q-i}) \\
&\subset& 
\Fp(v\hat{y}_A)^i\Tr_A^E\vphi^*(\res^E_{A'}(v)^{q-i}) \\ 
&\subset& 
(v\hat{y}_A)^iC^{q-i-1}M_{q-i}\\
&\subset& 
C^{q-i-1}v^iM_q
\end{eqnarray*}
by Lemma \ref{l63}.
\end{proof}

\begin{lem}\label{l66}
Suppose that $m \geq 0$ and $0 \leq i, q  \leq p-2$. 
Let $\vphi : A \lorarr A' \leq E$ be 
an isomorphism. Then the element 
$$\Tr_A^E\vphi^*(\res^E_{A'}(C^mv^q)(y')^i)$$
is contained in the following $\Fp$-subspace $W$:
$$\begin{array}{lc}
\mbox{{\rm condition}} & W \\ \hline 
q \leq i & \bCA \{v^iM_q\}\\
i<q,~m>0 & \bCA\{Vv^iM_q\} 
\end{array}$$
\end{lem}

\begin{proof} The proof of this lemma is similar to the proof 
of the previous lemma using Lemma \ref{l62} instead of Lemma \ref{l61}. 
Assume $q \leq i$. Then  
$$\Tr_A^E\vphi^*(\res^E_{A'}(C^mv^q)(y')^i) \subset   
\Fp(v\hat{y}_A)^q\Tr_A^E(u^{m(p-1)+(i-q)})$$
and the result follows from Lemma \ref{l62}. Next, assume 
$i<q$ and $m>0$. Then 
$$
\Tr_A^E\vphi^*(\res^E_{A'}(C^mv^q)(y')^i) 
\subset 
\Fp(v\hat{y}_A)^q\Tr_A^E(u^{(m-1)(p-1)+(p-1-(q-i))})$$
and the result follows from Lemma \ref{l62}. 
\end{proof}


\section{$A_p(E,E)$-submodules of $H^*(E)$}\label{sub}

Let $A_p(E,E)=\Fp \otimes A_{\bZ}(E,E)$ be the double Burnside 
algebra of $E$ over $\Fp$. 
In this section, we consider some $A_p(E,E)$-submodules of 
$H^*(E)$. 
Note that $A_p(E,E)=\Fp\Out(E) \oplus J(E)$. First 
we shall show that the ideal $J(E)$ is generated three kinds of 
bisets. 

\begin{defin}\label{d71} 
\rm{(1) Let $\mcl{I}_0$ be the $\Fp$-subspace of $A_p(E,E)$ spanned by all bisets corresponding to 
$E \geq A \stackrel{\vphi}{\lorarr} E$, $A \in \mcl{A}(E)$ and 
$\vphi$ is an injective morphism. \\
(2) Let $\mcl{I}_1$ be the $\Fp$-subspace of $A_p(E,E)$ spanned by all bisets corresponding to 
$\phi:E \lorarr E$ and $\phi(E)\in \mcl{A}(E)$. } 
\end{defin}
\begin{lem}\label{l72} 
The ideal $J(E)$ is generated by $\mcl{I}_0$, $\mcl{I}_1$ and the  
elements in $A_p(E,E)$ corresponding to the bisets which factor 
through the trivial subgroup.    
\end{lem}

\begin{proof}
Let $A \in \mcl{A}(E)$, $\psi:A \lorarr E$ and  
$\psi':E \lorarr E$ and assume that $|\psi(A)|=|\psi'(E)|=p$. 
It suffices to show that 
$\zeta_{A,\psi}$ and $\zeta_{E,\psi'}$ are contained 
in the ideal generated by $\mcl{I}_0$ and $\mcl{I}_1$. 

First consider $\zeta_{A,\psi}$. There exist an automorphism 
$\vphi_1:A \isoarr A$ and a surjective morphism 
$\psi_1:A \lorarr \psi(A)$ such that 
$\psi=\psi_1\vphi_1$ and $\Ker \psi_1= \la c \ra$. 
Then $\psi_1$ extends to $\phi_1:E \lorarr E$ such that 
$|\phi_1(E)|=p^2$ and we have 
$\psi = \phi_1\vphi_1$. Hence  
$$\zeta_{A,\psi}=\zeta_{E,\phi_1}\zeta_{A,\vphi_1} \in 
\mcl{I}_1\mcl{I}_0.$$
Next we consider $\zeta_{E,\psi'}$. 
There exist morphism 
$\phi_2, \phi_3:E \lorarr E$ such that 
$|\phi_2(E)|=|\phi_3(E)|=p^2$ and 
$\phi_3\phi_2=\psi'$. Then it follows that
$$\zeta_{E,\psi'}=\zeta_{E,\phi_3}\zeta_{E,\phi_2} \in 
\mcl{I}_1^2$$ 
and this completes the proof.
\end{proof}

Next, we consider the effect of $\mcl{I}_0$ on some subspaces  
of $H^*(E)$. 
Recall that $M_i=CS^i+T^i$ for $0 \leq i \leq p-2$ where 
$S^0=\Fp$ and $T^0=S^{p-1}$.

\begin{lem}\label{l73} {\rm(1)} 
If $0 \leq j \leq p-1$ and $0 \leq i \leq p-2$, then 
$$(C^jv^iM_i)\mcl{I}_0 \subset C^jv^iM_i$$
and 
$$v^i\mcl{I}_0\subset C^{i-1}M_i.$$
{\rm(2)}
If $j \geq 0$ and $0 \leq i,q \leq p-2$, then 
$$(C^jv^qM_i)\mcl{I}_0 \subset \bCA \{v^iM_q\}.$$
{\rm(3)}
If $j \geq 0$ and $1 \leq q \leq p-2$, then 
$$(C^jv^qM_0)\mcl{I}_0 \subset \bCA \{VM_q\}.$$
{\rm(4)} If $j \geq 0$ and $1 \leq i \leq p-2$, then 
$$(C^jVM_i)\mcl{I}_0 \subset \bCA \{v^iM_0\}.$$
\end{lem}

\begin{proof}
(1) follows from Lemma \ref{l65}. (2) and (3) follow from 
Lemma \ref{l66}. Let us prove (4). Since 
\begin{eqnarray*}
\res^E_A(C^jVM_i) & \subset & 
\Fp\res^E_A(C^jV)y^{p-1+i} \\
&=& \Fp\res^E_A(C^{j+1}V)y^i  \\
&=& \Fp \Tilde{D_2} \res^E_A(C^jy^i)
\end{eqnarray*}
for every $A \in \mcl{A}(E)$, 
it follows that $(C^jVM_i)\mcl{I}_0 \subset \bCA \{v^iM_0\}$ 
by Lemma \ref{l66}. 
\end{proof}

Now, we consider certain $A_p(E,E)$-submodules of 
$H^*(E)$.

\begin{thm}\label{t74}
$H^*(E)$ is a direct sum of the following $A_p(E,E)$-submodules.
\[
\bCA\{\Fp+S^{p-1}\} \tag{\ref{t74}.1}\]
\[
\bCA\{S^i +T^i + \Fp v^i + v^iS^{p-1}\}~(1 \leq i \leq p-2)
\tag{\ref{t74}.2} \]
\[
\bCA\{v^i(S^i \op T^i)\} ~ (1 \leq i \leq p-2) \tag{\ref{t74}.3}\] 
\[ 
\sum_{1 \leq i\ne q \leq p-2}\bCA\{v^q(S^i\op  T^i)\}. \tag{\ref{t74}.4} \]
\end{thm}

\begin{proof}
By Theorem \ref{t44}, as a $\GL_2(\bF_p)$-module, 
$H^*(E)$ is a direct sum of these submodules. 
We shall prove that these are in fact $A_p(E,E)$-submodules in 
Proposition \ref{p76}, \ref{p79}, \ref{p711} and \ref{p713} below. 
\end{proof}

Let us show that these modules are $A_p(E,E)$-submodules 
and analyze their detailed structure. We shall prove 
Corollary \ref{c77}, \ref{c710}, \ref{c712} and \ref{c714} using Proposition \ref{p82},  \ref{p84} and Corollary \ref{c93}. Note that 
above four corollaries are not used in section \ref{Q} and 
section \ref{A}. 
 
Let $W$ be an $\Fp\Out(E)$-submodule of $H^*(E)$. In order to 
prove that $W$ is an $A_p(E,E)$-submodule, 
it suffices to show that 
$W\mcl{I}_0$ and $W\mcl{I}_1$ are contained in $W$ by 
Lemma \ref{l72}. 
Note that $(\bDA x)\mcl{I}_0 \subset \bDA (x\mcl{I}_0)$ for any 
$x \in H^*(E)$. Moreover  
$(vS^i)\mcl{I}_1=(vT^i)\mcl{I}_1=0$ for $ i >0$ 
and $D_2\mcl{I}_1=0$ since $\vphi^*(d_2)=0$ for 
any surjective morphism $\vphi :E \lorarr A$. 

\begin{defin}\label{d75}
\rm{We set 
$$N_i=\sum_{j=0}^{p-1}C^jM_i$$ 
for $0 \leq i \leq p-2$.}
\end{defin}

First, we consider the module in (\ref{t74}.1). 

\begin{prop}\label{p76}
The following is a series of $A_p(E,E)$-submodules of $H^*(E)$, 
\begin{eqnarray*}
\bCA\{\Fp+S^{p-1}\} \supset \bCA \{M_0\} \supset 
\bCA \{VM_0\}.   
\end{eqnarray*}
Moreover we have a decomposition into  
$A_p(E,E)$-submodules, 
$$\bCA \{VM_0\}=\bDA\{D_2N_0\} \op \bDA\{D_2\}
 \op \bDA\{VS^{p-1}\}.$$
\end{prop}

\begin{proof}
First we show that $\bDA \{D_2N_0\}$ is an $A_p(E,E)$-submodule 
of $H^*(E)$. 
We have 
$\bDA \{D_2N_0\}\mcl{I}_1=0$. On the other hand, 
since $C^jM_0\mcl{I}_0 \subset C^jM_0$ for $0 \leq j \leq p-1$ 
by Lemma \ref{l73}(1), 
$\bDA \{D_2N_0\}\mcl{I}_0 \subset \bDA\{D_2N_0\}$. 

Next, since $\bDA\{D_2\}\mcl{I}_0=\bDA\{D_2\}\mcl{I}_1=0$, 
$\bDA\{D_2\}$ is an $A_p(E,E)$-submodule. Similarly, 
$\bDA\{VS^{p-1}\}\mcl{I}_1=0$. 
Moreover $\bDA\{VS^{p-1}\}\mcl{I}_0=0$ 
by Lemma \ref{l64} and $\bDA\{VS^{p-1}\}$ is an 
$A_p(E,E)$-submodule. 
We have a direct sum decomposition as an $\Fp$-vector space, 
\begin{eqnarray*}
\bCA \{VM_0\}&=&\bDA\{\sum_{j=0}^p (C^j VM_0)\} \\
 &=& \bDA \{C(\sum_{j=0}^{p-1} (C^jVM_0)\}
 \op \bDA\{VM_0\} \\
&=& \bDA\{D_2N_0\} \op \bDA\{D_2\} \op \bDA\{VS^{p-1}\} 
\end{eqnarray*}
and this is a decomposition as an $A_p(E,E)$-module. 

Next we consider $\bCA\{M_0\}$. By Proposition \ref{p42}, 
$\bCA \{M_0\}\mcl{I}_1 \subset \bCA \{M_0\}$. 
By Lemma \ref{l73}(2), 
$$\bCA \{M_0\}\mcl{I}_0=\bDA\{\sum_{j=0}^p C^jM_0\} \mcl{I}_0
\subset \bCA \{M_0\}.$$
Hence $\bCA \{M_0\}$ is an $A_p(E,E)$-submodule. 
Finally we consider $\bCA\{\Fp+S^{p-1}\}$. By Proposition \ref{p42}, 
$(\bCA \{\Fp+S^{p-1}\}^+) \mcl{I}_1 \subset \bCA \{M_0\}$. 
Note that, as $\Fp$-vector spaces, 
$\bCA\{\Fp+S^{p-1}\}=\Fp[D_1]\op \bCA \{M_0\}$. Since  
$\Fp[D_1]\mcl{I}_0=0$, we have 
$\bCA\{\Fp+S^{p-1}\}\mcl{I}_0 \subset \bCA \{M_0\}$. 
\end{proof}

\begin{cor}\label{c77}
{\rm(1)}   
$\bCA\{\Fp+S^{p-1}\}=\Fp[D_1] \op \bCA\{M_0\}$
as $\GL_2(\Fp)$-modules and 
$$\bCA\{\Fp+S^{p-1}\}^+/\bCA \{M_0\} \simeq 
\bigoplus  
S(E,E,\Fp)$$ as $A_p(E,E)$-modules. \\
{\rm(2)} $\bCA\{M_0\}=\Fp[C]\{M_0\}\op 
\bCA\{VM_0\}$ as $\GL_2(\Fp)$-modules and 
$$\bCA\{M_0\}/\bCA\{VM_0\} \simeq \bigoplus  S(E,Q,U_0)$$ 
as $A_p(E,E)$-modules, where $Q \leq P$, $|Q|=p$ and $U_0$ is the trivial 
$\Fp\Out(Q)$-module.   \\
{\rm(3)} As $A_p(E,E)$-modules, 
$$\bDA\{D_2N_0\} \simeq \bigoplus S(E,A,S(A)^{p-1})$$
$$\bDA\{D_2\}  \simeq \bigoplus S(E,E, \Fp)$$
and  
$$\bDA\{VS^{p-1}\} \simeq \bigoplus S(E,E, S^{p-1}).$$
\end{cor}

\begin{proof}
(1) Since $\bCA\{\Fp +S^{p-1}\}\mcl{I}_0$ and 
$(\bCA\{\Fp +S^{p-1}\}^+)\mcl{I}_1$ are contained 
in $\bCA \{M_0\}$ as in the proof of Proposition \ref{p76}, 
it follows that 
$$(\bCA\{\Fp+S^{p-1}\}^+/\bCA\{M_0\})J(E)=0.$$
Since $\GL_2(\Fp)$ acts on  $D_1$ trivially, we have 
$$\bCA\{\Fp+S^{p-1}\}^+/\bCA\{M_0\} \simeq \bigoplus 
S(E,E,\Fp)$$ 
(2) This follows from Proposition \ref{p82}.\\
(3) By Corollary \ref{c93}, 
$$\bDA\{D_2N_0\} \simeq \bigoplus S(E,A,S(A)^{p-1}).$$
Since $\bDA\{D_2\}\mcl{I}_0=\bDA\{D_2\}\mcl{I}_1=0$ 
and $\GL_2(\Fp)$ acts 
on $D_2$ trivially, it follows that 
$$\bDA\{D_2\} \simeq \bigoplus S(E,E, \Fp). $$
Similarly, since $$\bDA\{VS^{p-1}\}\mcl{I}_0=\bDA\{VS^{p-1}\}\mcl{I}_1=0,$$
it follows that 
$$\bDA\{VS^{p-1}\} \simeq \bigoplus S(E,E, S^{p-1}).$$
\end{proof}

Next, we consider the module in (\ref{t74}.2). 

\begin{lem}\label{l78} 
Let $1\leq i \leq p-2$. \\
{\rm(1)} $\bCA\{S^i+T^i\}\mcl{I}_0\subset \bCA\{v^iM_0\}$. \\
{\rm(2)} $\bCA\{S^i+T^i\}\mcl{I}_1 \subset \Fp[C]\{S^i\}$.
\end{lem}

\begin{proof}
(1) First, note that 
$$\bCA\{S^i+T^i\}\subset \bDA \{S^i\}+\bDA[C]\{M_i\}.$$
Since $\bDA\{S^i\}\mcl{I}_0=0$ and $\bDA[C]\{M_i\}
\mcl{I}_0 \subset \bCA\{v^iM_0\}$ by Lemma \ref{l73}(2), 
the proof is completed. \\
(2) We have that 
$\bCA \{V(S^i+T^i)\}\mcl{I}_1=0$. Hence it follows that 
$$\bCA\{S^i+T^i\}\mcl{I}_1\subset \Fp[C]\{S^i+T^i\}\mcl{I}_1\subset 
\Fp[C]\{S^i\}$$
by Corollary \ref{c49} and the proof is completed. 
\end{proof}

\begin{prop}\label{p79} 
Let $1 \leq i\leq p-2$. 
Let $Z_i=\Fp[C]\{S^i\}+\bCA\{v^iM_0+VM_i\}$. Then we have the  following $A_p(E,E)$-submodules: 
$$\begin{array}{c}
          \bCA \{S^i+T^i+\Fp v^i+v^iS^{p-1}\}         \\
             |                                    \\
       \bCA\{S^i+T^i+v^iM_0\}    \\
\diagup~ \diagdown \\
\Fp[V]\{VS^i\}+Z_i \qquad   \Fp[C]\{T^i\}+Z_i \quad \\
\diagdown ~ \diagup\\
Z_i\\
| \\
\bCA\{v^iM_0+VM_i\}
\end{array}$$
\end{prop}

\begin{proof} 
First, we consider $\bCA \{v^iM_0+VM_i\}$. We have 
$\bCA\{v^iM_0+VM_i\}\mcl{I}_1=0$. 
Since $\bCA=\bDA\{1,C,\dots, C^p\}$, in order to show that 
$\bCA\{v^iM_0+VM_i\}\mcl{I}_0 
\subset  \bCA\{v^iM_0+VM_i\}$, it suffices to show that 
$C^j\{v^iM_0+VM_i\}\mcl{I}_0 
\subset  \bCA\{v^iM_0+VM_i\}$ for $j \geq 0$. 
By Lemma \ref{l73}(3), we have 
$(C^jv^iM_0)\mcl{I}_0 \subset \bCA\{VM_i\}$. 
On the other hand, 
$(C^jVM_i)\mcl{I}_0 \subset \bCA\{v^iM_0\}$ 
by Lemma \ref{l73}(4).  
This shows that 
$\bCA\{v^iM_0+VM_i\}$ is an $A_p(E,E)$-submodule. 
 
Next, consider $Z_i=\Fp[C]\{S^i\}+\bCA\{v^iM_0+VM_i\}$. Since 
$\res^E_A(C^jS^i)=\Fp y^{j(p-1)+i}$, 
it follows that 
$(\Fp[C]\{S^i\})\mcl{I}_1\subset \Fp[C]\{S^i\}$ by Corollary \ref{c49}. 
On the other hand, we have that 
$\Fp[C]\{S^i\}\mcl{I}_0 \subset \bCA \{v^iM_0\}$ by Lemma \ref{l73}(2). 
Hence $Z_i$ is an $A_p(E,E)$-submodule. 

Next we consider 
$\bCA\{S^i+T^i+v^iM_0\}$, $\Fp[V]\{VS^i\}+Z_i$ and 
$\Fp[C]\{T^i\}+Z_i$. Since these are $\GL_2(\Fp)$-submodules, 
it suffices to show that 
$\bCA\{S^i+T^i+v^iM_0\}J(E)\subset Z_i.$ 
But this follows from Lemma \ref{l78} since 
$\bCA\{v^iM_0\} \subset Z_i$. 

Finally, we consider $\bCA \{S^i+T^i+\Fp v^i+v^iS^{p-1}\}$. 
Note that 
$$\bCA\{S^i+T^i+\Fp v^i+v^iS^{p-1}\}
=\bCA\{v^i\}+\bCA\{S^i+T^i+v^iM_0\}.$$
We claim that 
$\bCA\{v^i\}\mcl{I}_0 \subset \bCA\{S^i+T^i\}$. 
It suffices to show that $(C^jv^i)\mcl{I}_0 \subset \bCA\{S^i+T^i\}$ 
for $j \geq 0$. If $j=0$, then 
$v^i\mcl{I}_0 \subset C^{i-1}M_i$ by Lemma \ref{l73}(1). 
If $j>0$, then 
$$(C^jv^i)\mcl{I}_0 \subset 
(C^{j-1}v^iM_0)\mcl{I}_0 \subset \bCA\{VM_i\} 
\subset \bCA\{S^i+T^i\}$$
by Lemma \ref{l73}(3). 
On the other hand, 
$\bCA\{v^i\}\mcl{I}_1 \subset \bCA\{S^i,T^i\}$ by 
Proposition \ref{p42}. Hence 
$$\bCA\{S^i+T^i+v^i+\Fp v^i+S^{p-1}\}J(E)\subset 
\bCA\{S^i+T^i+v^iM_0\}$$
and this completes the proof. 
\end{proof}

\begin{cor}\label{c710} 
Let $1 \leq i \leq p-2$ and set 
$Z_i=\Fp[C]\{S^i\}+\bCA\{v^iM_0+VM_i\}$ as in Proposition \ref{p79}. 
The we have the following.\\ 
{\rm(1)} $\bCA\{S^i+T^i+\Fp v^i+v^iS^{p-1}\}=
\Fp[V]\{v^i\}\op \bCA
\{S^i+T^i+v^iM_0\}$ as $\GL_2(\bF_p)$-modules and  $$\bCA\{S^i+T^i+\Fp v^i+v^iS^{p-1}\}/\bCA\{S^i+T^i+v^iM_0\}
\simeq \bigoplus S(E,E,{\det}^i)$$
as $A_p(E,E)$-modules.  
\\
{\rm(2)} $\bCA\{S^i+T^i+v^iM_0\}=\Fp[C]\{T^i\}\op 
\Fp[V]\{VS^i\}\op Z_i$ as $\Fp$-spaces and we have 
$\bCA\{S^i+T^i+v^iM_0\}J(E)\subset Z_i$. In particular, 
$$\bCA\{S^i+T^i+v^iM_0\}/(\Fp[C]\{T^i\}+Z_i) \simeq 
(\Fp[V]\{VS^i\}+Z_i)/Z_i \simeq \bigoplus S(E,E,S^i)$$ 
and 
$$\bCA\{S^i+T^i+v^iM_0\}/(\Fp[V]\{VS^i\}+Z_i) \simeq 
(\Fp[C]\{T^i\}+Z_i)/Z_i \simeq \bigoplus S(E,E,\bar{T}^i)$$
as $A_p(E,E)$-modules, where $\bar{T}^i=(CS^i+T^i)/CS^i$. 
\\
{\rm(3)} $Z_i/\bCA\{v^iM_0+VM_i\} \simeq \op S(E,Q,U_i)$ 
where $Q \leq E$, $|Q|=p$ and $U_i$ is a simple 
$\Fp \Out(Q)$-module. (See the first part of section 5.) \\
{\rm(4)} Every composition factor of $\bCA\{v^iM_0+VM_i\}$ 
as an $A_p(E,E)$-module has minimal subgroup $E$. 
\end{cor}

\begin{proof}
(1) Since there is a decomposition into 
$\Fp$-vector subspaces
$$\bCA\{\Fp+S^{p-1}\}=\Fp[V]\op \bCA\{C\} \op \bCA\{S^{p-1}\},$$
we have 
$$\bCA\{S^i+T^i+\Fp v^i+v^iS^{p-1}\}=
\Fp[V]\{v^i\}\op \bCA
\{S^i+T^i+v^iM_0\}.$$
Moreover 
$$(\bCA\{S^i+T^i+\Fp v^i+v^iS^{p-1}\}/\bCA\{S^i+T^i+v^iM_0\})
J(E)=0$$
by the proof of Proposition \ref{p79}. Hence 
the result follows since $\GL_2(\Fp)$ acts on $v^i$ 
as ${\det}^i$.  \\
(2) Since there exist decompositions into 
$\Fp$-subspaces, 
$$\bCA\{S^i\}=\Fp[C]\{S^i\}\op \Fp [V]\{VS^i\}\op 
\bCA\{VCS^i\}$$
$$\bCA\{T^i\}=\Fp[C]\{T^i\}\op \bCA \{VT^i\},$$
we have 
$$\bCA\{S^i+T^i\}=\Fp[C]\{T^i\} \op \Fp[V]\{VS^i\} \op 
\Fp[C]\{S^i\}\op \bCA\{VM_i\}$$
and 
$$\bCA\{S^i+T^i+v^iM_0\}=
\Fp[C]\{T^i\} \op \Fp[V]\{VS^i\}\op Z_i.$$
Hence the result follows since 
$$\bCA\{S^i+T^i+v^iM_0\}J(E) \subset Z_i$$
by the proof of Proposition \ref{p79}. \\
(3) This follows from Proposition \ref{p84}. \\
(4) This follows from Proposition \ref{p82}, \ref{p84} and  
Corollary \ref{c93}.
\end{proof}

Next, we consider the submodule in (\ref{t74}.3).

\begin{prop}\label{p711} 
Let $1 \leq i\leq p-2$.  
The $\Fp$-subspace $\bCA\{v^iS^i+v^iT^i\}$ is an $A_p(E,E)$-submodule of $H^*(E)$ and 
we have the following decomposition into  
$A_p(E,E)$-submodules, 
$$\bCA\{v^iS^i+v^iT^i\}=\bDA\{v^iN_i\}\op \bDA
\{v^i(S^i+VT^i)\}.$$
\end{prop}

\begin{proof}
Since
$$\bCA \{v^iS^i+v^iT^i\}=\bDA\{v^iN_i\} \op \bDA 
\{v^i(S^i+VT^i)\}$$
as $\Fp$-vector spaces, it suffices to show that $\bDA\{v^iN_i\}$ and $\bDA\{v^i(S^i+VT^i)\}$ 
are $A_p(E,E)$-submodules. First, 
these are $GL_2(\Fp)$-submodules since 
$V(CS^i+T^i) =D_2S^i+VT^i$. 
We have $\bCA \{v^iS^i+v^iT^i\}\mcl{I}_1=0$. 
Since $C^jv^iM_i\mcl{I}_0 \subset C^jv^iM_i$ for 
$0 \leq j\leq p-1$ by Lemma \ref{l73}(1), we have 
$\bDA\{v^iN_i\}\mcl{I}_0 \subset \bDA\{v^iN_i\}$. 
On the other hand, $\bDA\{v^i(S^i+VT^i)\}\mcl{I}_0=0$ 
by Lemma \ref{l64} and this completes the proof. 
\end{proof}

\begin{cor}\label{c712}
{\rm(1)} As $A(E,E)$-modules, 
$$\bDA\{v^iN_i\} \simeq \bigoplus S(E,A,S^{p-1}(A)\ot {\det}^i).$$
{\rm(2)} As an $A(E,E)$-module, every composition 
factor of $\bDA \{v^i(S^i+VT^i)\}$ has minimal subgroup $E$. 
\end{cor}

\begin{proof}
(1) This follows from Corollary \ref{c93}. \\
(2) Since $\bDA\{v^i(S^i+VT^i)\}J(E)=0$ as in the proof of 
Proposition \ref{p711}, every composition factor of 
$\bDA\{v^i(S^i+VT^i)\}$ has minimal subgroup $E$.   
\end{proof}

Finally, let us consider the submodule in (\ref{t74}.4).

\begin{prop}\label{p713}
The $\Fp$-subspace 
$$\sum_{1 \leq i\ne q \leq p-2}\bCA\{v^qS^i+v^qT^i\}$$
is an $A_p(E,E)$-submodule of $H^*(E)$. 
\end{prop}

\begin{proof}
Let $W=\sum_{1 \leq i\ne q \leq p-2}\bCA\{v^qS^i+v^qT^i\}$. 
We have $W\mcl{I}_1=0$. Since $\bCA=\sum_{j=0}^p \bDA\{C^j\}$, 
it suffices to show that $C^j(v^qS^i+v^qT^i)\mcl{I}_0\subset 
W$ for $0 \leq j \leq p$. By Lemma \ref{l65}, 
$C^j(v^qS^i+v^qT^i)\mcl{I}_0\subset \bCA \{v^iM_q\} \subset W$, 
and this completes the proof. 
\end{proof}

\begin{cor}\label{c714}
As an $A(E,E)$-module, every composition factor of 
$$\sum_{1 \leq i\ne q \leq p-2}\bCA\{v^qS^i+v^qT^i\}$$
has minimal subgroup $E$.
\end{cor}

\begin{proof}
This follows from Proposition \ref{p82}, \ref{p84} and  
Corollary \ref{c93}.
\end{proof}


\section{Simple modules with cyclic minimal subgroup}\label{Q}

In this section, we consider simple $A_p(E,E)$-modules induced 
from cyclic subgroup of order $p$. 
Recall that  $M_i=CS^i+T^i$ for $0 \leq i \leq p-2$ where 
$S^0=\Fp$ and $T^0=S^{p-1}$. 

Let $Q\leq E$ be a subgroup of order $p$. 
Let $U_i=\Fp u_i ~(0 \leq i \leq p-2)$ be the simple right 
$\Fp\Out(Q)$-module defined by 
$u_i\sigma =\lam^i u_i$ where $\Out(Q)=\la \sigma \ra$ and  
$\bF_p^{\times}=\la \lam \ra$. 
Let $\eta$ be a generator of $H^2(Q)$. 
If $n=m(p-1)+i$ where $m \geq 0$ and $0 \leq i \leq p-2$,  
then $H^{2n}(Q)\simeq U_i$ as 
$\Fp\Out(Q)$-modules. 

First, we consider the simple $A_p(E,E)$-module corresponding to 
$U_0$, namely, $S(E,Q,U_0)$.

\begin{lem}\label{l81} 
Let $n=m(p-1)>0$. 
Then we have 
$$H^{2n}(Q)A(E,Q)+\bCA \{VM_0\}=C^{m-1}M_0+\bCA \{VM_0\}.$$
\end{lem}

\begin{proof}
$H^{2n}(Q)A(E,Q)$ is generated by 
$\displaystyle{\sum_{\phi:E\thrarr Q}\phi^*H^{2n}(Q)}$ and 
$$\sum_{A \in \mcl{A}(E)}\sum_{\vphi:A \thrarr Q}
\Tr_A^E\vphi^*(H^{2n}(Q)).$$ 
The former $\Fp$-subspace is equal to $C^{m-1}S^{p-1}$ by 
Corollary \ref{c49}. 
On the other hand, if $\vphi:A \thrarr Q$ then 
$\vphi^*(\eta)$ is $\lam y$ or $E$-conjugate of $\lam u$ 
for some $\lam \in \bF_p$.  Hence the later $\Fp$-subspace 
is equal to $\sum_{A\in \mcl{A}(E)}\Fp\Tr_A^E(u^n).$  
By Lemma \ref{l62}, 
$$\Tr_A^E(u^n) \equiv C^{m-1}\Tr_A^E(u^{p-1}) 
\pmod{\bCA\{VM_0\}}.$$
Since $C^{m-1}\Tr_A^E(u^{p-1})$, $A \in \mcl{A}(E)$ is a basis of 
$C^{m-1}M_0$, it follows that 
$$\sum_{A \in \mcl{A}(E)}\Fp\Tr_A^E(u^n)+\bCA\{VM_0\}=
C^{m-1}M_0+\bCA\{VM_0\}$$
and this completes the proof.   
\end{proof}  

\begin{prop}\label{p82} As $A_p(E,E)$-modules, we have 
$$\bCA\{M_0\}/\bCA\{VM_0\} \simeq \bigoplus S(E,Q,U_0).$$
Moreover, 
$H^*(E)/\bCA\{M_0\}$ and $\bCA\{VM_0\}$ have no simple subquotient 
module which is isomorphic to $S(E,Q,U_0)$.   
\end{prop}

\begin{proof} First, note that $\bCA\{M_0\}$ and $\bCA\{VM_0\}$ are 
$A_p(E,E)$-submodules by Proposition \ref{p76}. 
Let $F_m=H^{2m(p-1)}(Q)A(-,Q)$ for $m \geq 1$. 
By Lemma \ref{l81}, 
$$F_m(E)+\bCA\{VM_0\}
=C^{m-1}M_0+\bCA\{VM_0\}.$$
In particular, 
$$(\bigcap_{R \leq E, |R|=p}\Ker \res^E_R) \cap F_m(E)
=\bCA\{VM_0\}\cap F_m(E).$$
Hence $F_m(E)/(F_m(E)\cap \bCA\{VM_0\}) \simeq S(E,Q,U_0)$ 
by Lemma \ref{l31}, 
and we have 
\begin{eqnarray*}
\bCA\{M_0\}/\bCA\{VM_0\} &\simeq&
 (\sum_{m \geq 0} F_m(E)+\bCA\{VM_0\})/\bCA\{VM_0\} \\ 
&\simeq& \bigoplus F_m(E)/(F_m(E)\cap \bCA\{VM_0\})\\
&\simeq& \bigoplus S(E,Q,U_0).
\end{eqnarray*}

Let $F=\op_{m\geq 1}F_m$. Since 
$H^*(Q)/F(Q)$ has no simple subquotient module 
which is isomorphic to 
$S(Q,Q,U_0)$, it follows that 
$H^*(E)/F(E)$ has no simple subquotient module 
which is isomorphic to 
$S(E,Q,U_0)$ by Corollary \ref{c33}. 
Since $F(E) \subset \bCA \{M_0\}$, it 
follows that 
$H^*(E)/\bCA \{M_0\}$ has no simple subquotient module 
which is isomorphic to $S(E,Q,U_0)$. 
Moreover, since the restriction of 
$\bCA \{VM_0\}$ to any subgroup of order $p$ is zero, 
we have 
$(\bCA \{VM_0\})A_p(Q,E)=0$. 
Hence $\bCA \{VM_0\}$ has no simple subquotient module 
which is isomorphic to $S(E,Q,U_0)$ by Lemma \ref{l32}. 
\end{proof}

Next, we consider the simple $A_p(E,E)$-module corresponding to 
$U_i$ for $1 \leq i \leq p-2$, namely, 
$S(E,Q,U_i)$.  

\begin{lem}\label{l83}
Let $1 \leq i \leq p-2$ and $n=m(p-1)+i$. 
Then we have 
$$H^{2n}(Q)A(E,Q)+\bCA\{v^iM_0+VM_i\}=
C^m S^i+\bCA\{v^iM_0+VM_i\}.$$
\end{lem}

\begin{proof} 
By Corollary \ref{c49}, 
$H^{2n}(Q)A(E,Q)$ is spanned by $C^mS^i$ and 
$$\sum_{A \in \mcl{A}(E)}\sum_{\vphi:A\thrarr Q}\Tr_A^E\vphi^*(H^{2n}(Q)).$$
Since 
$\vphi^*(\eta)$ is $\lam y$ or $E$-conjugate of $\lam u$ 
for some $\lam \in \bF_p$, 
it follows that 
$$H^{2n}(Q)A(E,Q)+\bCA\{v^iM_0+VM_i\}=
C^mS^i+\bCA\{v^iM_0+VM_i\}$$
by Lemma \ref{l62} and this completes the proof.  
\end{proof}

\begin{prop}\label{p84} 
Let $1 \leq i \leq p-2$. As $A_p(E,E)$-modules, we have 
$$(\Fp[C]\{S^i\} +\bCA\{v^iM_0+VM_i\})/\bCA\{v^iM_0+VM_i\} \simeq 
\bigoplus S(E,Q,U_i).$$
Moreover, 
$H^*(E)/(\Fp[C]\{S^i\} +\bCA\{v^iM_0+VM_i\})$ and 
$\bCA\{v^iM_0+VM_i\}$ have no simple subquotient 
module which is isomorphic to $S(E,Q,U_i)$.   
\end{prop}

\begin{proof} We set 
$Z_i=\Fp[C]\{S^i\} +\bCA\{v^iM_0+VM_i\}$ as in 
Proposition \ref{p79}. 
Then $Z_i$ and $\bCA\{v^iM_0+VM_i\}$ are $A_p(E,E)$-submodules by Lemma \ref{p79}. 
Let $$F_{m,i}=H^{2m(p-1)+i}(Q)A(-,Q)$$ for $m \geq 0$. 
By Lemma \ref{l83}, 
$$F_{m,i}(E)+\bCA\{v^iM_0+VM_i\}
=C^mS^i+\bCA\{v^iM_0+VM_i\}.$$
In particular, 
$$(\bigcap_{R \leq E, |R|=p}\Ker \res^E_R) \cap F_{m,i}(E)
=\bCA\{v^iM_0+VM_i\}\cap F_{m,i}(E).$$
Hence $F_{m,i}(E)/(F_{m,i}(E)\cap \bCA\{v^iM_0+VM_i\}) \simeq S(E,Q,U_i)$ 
by Lemma \ref{l31}, 
and we have 
\begin{eqnarray*}
Z_i/\bCA\{v^iM_0+VM_i\} 
&\simeq&
 (\sum_{m \geq 0} F_{m,i}(E)+\bCA\{v^iM_0+VM_i\})/
\bCA\{v^iM_0+VM_i\} \\
&\simeq& 
\bigoplus_{m \geq 0} F_{m,i}(E)/
(F_{m,i}(E)\cap \bCA\{v^iM_0+VM_i\})\\
& \simeq & \bigoplus S(E,Q,U_i).
\end{eqnarray*}

Let $F=\op_{m\geq 0}F_{m,i}$. Since 
$H^*(Q)/F(Q)$ has no simple subquotient module 
which is isomorphic to 
$S(Q,Q,U_i)$, it follows that 
$H^*(E)/F(E)$ has no simple subquotient module 
which is isomorphic to 
$S(E,Q,U_i)$ by Corollary \ref{c33}. 
Since $F(E) \subset Z_i$, 
it follows that 
$H^*(E)/Z_i$ has no simple subquotient module 
which is isomorphic to $S(E,Q,U_i)$. 
Moreover, since the restriction of 
$\bCA \{v^iM_0+VM_i\}$ to any subgroup of order $p$ is zero, 
we have 
$(\bCA \{v^iM_0+VM_i\})A_p(Q,E)=0$. 
Hence $\bCA \{v^iM_0+VM_i\}$ has no simple subquotient module 
which is isomorphic to $S(E,Q,U_i)$ by Lemma \ref{l32}. 

\end{proof}


\section{Simple modules with minimal subgroup $A$}\label{A}
In this section, we consider $A_p(E,E)$-modules induced 
from $A_p(A,A)$-modules where $A$ is a maximal elementary 
abelian $p$-subgroup of $E$.  

\begin{lem}\label{l91}
Let $i,j \geq0$, $0 \leq m \leq p-2$ and 
$0 \leq n \leq p-1$. If $j+m>0$, then   
$\Tilde{D_1}^i\Tilde{D_2}^jd_2^mW_n A_p(E,A)$ is a simple 
$A_p(E,E)$-submodule of $H^*(E)$ and isomorphic to 
$S(E,A,S(A)^{p-1}\ot {\det}^m)$. 
\end{lem}

\begin{proof}
Let $F=(\Tilde{D_1}^i\Tilde{D_2}^jd_2^mW_n)A_p(-,A)$. 
Note that 
$$\Tilde{D_1}^i\Tilde{D_2}^jd_2^mW_n \simeq 
S(A,A,S(A)^{p-1}\ot {\det}^m)$$
by Proposition \ref{p57}. 
Since $F(E)\ne 0$, it suffices to show that 
$$\bigcap_{\phi \in A_p(A,E)}\Ker F(\phi)=0$$ 
by Lemma \ref{l31}. If $\al \in 
\cap_{\phi \in A_p(A,E)}\Ker F(\phi)$ then 
$\res^E_{A'}(\al)=0$ for every maximal elementary abelian 
subgroup $A'$ of $E$. Then $\al=0$ by Quillen's theorem \cite{Q}.  
\end{proof}

\begin{prop}\label{p92}
Let $i,j \geq0$, $0 \leq m \leq p-2$ and 
$0 \leq n \leq p-1$. If $j+m>0$, then 
$$(\Tilde{D_1}^i\Tilde{D_2}^jd_2^mW_n) A_p(E,A)=
D_1^iD_2^jv^mC^nM_m$$ where 
$M_m=CS^m+T^m$. 
\end{prop}

\begin{proof}
First, we claim that $D_1^iD_2^jv^mC^nM_m$ is an $A_p(E,E)$-submodule of $H^*(E)$. We have 
$(D_1^iD_2^jv^mC^nM_m)\mcl{I}_1=0$. Since 
$v^mC^nM_m \mcl{I}_0\subset v^mC^nM_m$ by 
Lemma \ref{l73}(1), $D_1^iD_2^jv^mC^nM_m$ is an $A_p(E,E)$-submodule. 

By Lemma \ref{l61}(3), 
$\{D_1^iD_2^jv^m\hat{y}_{A'}^mC^n\Tr(u^{p-1})\}_{A' \in \mcl{A}(E)}$ 
is a basis of $D_1^iD_2^jv^mC^nM_m$. Since 
$C^n\Tr_{A'}^E(u^{p-1})=\Tr_{A'}^E(u^{(n+1)(p-1)})$ by 
Lemma \ref{l61}(2), we have 
$$D_1^iD_2^jv^m\hat{y}_{A'}^mC^n\Tr(u^{p-1})
=\Tr_{A'}^E
(\Tilde{D_1}^i\Tilde{D_2}^jd_2^mu^{(n+1)(p-1)})$$
and it follows that 
$$D_1^iD_2^jv^mC^nM_m 
\subset 
(\Tilde{D_1}^i\Tilde{D_2}^jd_2^mW_n)A_p(E,A)+
(\Tilde{D_1}^i\Tilde{D_2}^jd_2^{m+1}H^*(A))A_p(E,A).$$
Since 
\begin{eqnarray*}
(\Tilde{D_1}^i\Tilde{D_2}^jd_2^{m+1}H^*(A))A_p(E,A)
&\subset& 
\sum_{A' \in \mcl{A}(E)}\sum_{\vphi:A' \isoarr A}
\Tr_{A'}^E\vphi^*(\Tilde{D_1}^i\Tilde{D_2}^jd_2^{m+1}H^*(A)) \\
&\subset& 
D_1^iD_2^jv^{m+1}S^{m+1}H^*(E),
\end{eqnarray*}
we have 
$$D_1^iD_2^jv^mC^nM_n \cap 
(\Tilde{D_1}^i\Tilde{D_2}^jd_2^{m+1}H^*(A))A_p(E,A)
=0. $$ 
Then it follows that 
\begin{eqnarray*}
 & &(\Tilde{D_1}^i\Tilde{D_2}^jd_2^mW_n) A_p(E,A)+
(\Tilde{D_1}^i\Tilde{D_2}^jd_2^{m+1}H^*(A)) A_p(E,A)\\
&=&
D_1^iD_2^jv^mC^nM_m
+(\Tilde{D_1}^i\Tilde{D_2}^jd_2^{m+1}H^*(A))A_p(E,A). 
\end{eqnarray*}
and 
$$D_1^iD_2^jv^mC^nM_n \simeq  (\Tilde{D_1}^i\Tilde{D_2}^jd_2^mW_n)A_p(E,A)$$
since $(\Tilde{D_1}^i\Tilde{D_2}^jd_2^mW_n)A_p(E,A)$ is 
a simple $A_p(E,E)$-module by Lemma \ref{l91}. 

Next,  we prove that 
$(\Tilde{D_1}^i\Tilde{D_2}^jd_2^{m+1}H^*(A))A_p(E,A)$ does not 
have a submodule which is isomorphic to 
$D_1^iD_2^jv^mC^nM_m$. In particular this implies that 
$$(\Tilde{D_1}^i\Tilde{D_2}^jd_2^mW_n) A_p(E,A)=
D_1^iD_2^jv^mC^nM_m.$$
Assume that there exists an $A_p(E,E)$-submodule of 
$(\Tilde{D_1}^i\Tilde{D_2}^jd_2^{m+1}H^*(A))A_p(E,A)$ 
isomorphic to $D_1^iD_2^jv^mC^nM_m$. Then some 
submodules of 
$D_1^iD_2^jv^{m+1}S^{m+1}H^*(E)$ 
is isomorphic to $D_1^iD_2^jv^mC^nM_m$. 
Then, as a $\GL_2(\bF_p)$-module, $vS^{m+1}H^{2r}(E)$ 
has a submodule which is isomorphic to 
$S^m$ where 
\begin{eqnarray*}
r&=&\frac{1}{2}(\deg C^nM_m-\deg v S^{m+1}) \\
 &=&n(p-1)+(p+m-1)-(p+m+1)=n(p-1)-2.
\end{eqnarray*}
But this contradicts to Lemma \ref{l46} and so the proof is completed.  
\end{proof}

\begin{cor}\label{c93}
{\rm(1)} We have 
$$\bDA \{D_2N_0\}\simeq \bigoplus 
S(E,A,S(A)^{p-1})$$
and the quotient module 
$$H^*(E)/\bDA \{D_2N_0\}$$
has no simple subquotient module which is isomorphic to 
$S(E,A,S(A)^{p-1})$. \\
{\rm(2)} Assume that $1\leq m \leq p-2$. Then 
$$\bDA\{v^mN_m\}\simeq \bigoplus 
S(E,A,S(A)^{p-1}\ot {\det}^m)$$
and the quotient module 
$$H^*(E)/\bDA \{v^mN_m\}$$
has no simple subquotient module which is isomorphic to 
$S(E,A,S(A)^{p-1}\ot{\det}^m)$. 
\end{cor}

\begin{proof} (1) 
By Lemma \ref{l91} and Proposition \ref{p92}, 
$\bDA \{D_2N_0\}$ is isomorphic to a direct sum of 
$S(E,A, S(A)^{p-1})$. Since 
$H^*(A)/\Tilde{\bDA} \{\op_{n=0}^{p-1}\Tilde{D_2}W_n\}$ 
has no simple subquotient which is isomorphic to 
$S(A,A,S(A)^{p-1})$ by Proposition \ref{p57}, 
it follows that $H^*(E)/\bDA \{D_2N_0\}$
has no simple subquotient module which is isomorphic to 
$S(E,A,S(A)^{p-1})$ by Corollary \ref{c33}. 
This completes the proof of (1) since 
$$\Tilde{\bDA} \{\bigoplus_{n=0}^{p-1}\Tilde{D_2}W_n\}
A_p(E,A)= \bDA\{D_2N_0\}$$
by Proposition \ref{p92}. The proof of (2) is similar to the 
proof of (1). 
\end{proof}

\begin{rem}\label{r94}
\rm{
We consider $d_2^q S(A)^i$ for $0\leq i \leq p-2$ and 
$q>0$.  
This is a simple $A_p(A,A)$-submodule of $H^*(A)$ with 
minimal subgroup $A$. 
Let $F=d_2^q S(A)^iA(-,A)$. Since 
$$\Tr_A^E(d_2^q S(A)^i)=0,$$
we have that $F(E)=0$ and in particular, 
$S_{A, S(A)^i \ot {\det}^q}(E)=0$. Hence there exists no 
simple $A_p(E,E)$-module corresponding to $S(A)^i \ot \det^q$ 
for $0 \leq i,q \leq p-2$. } 
\end{rem}


\section{Cohomology of stable summands}\label{main}

By results in section \ref{Q} and section \ref{A},  we have the following classification of simple $A_p(E,E)$-modules. 

\begin{prop}[\cite{DP}]\label{p101}
Up to isomorphisms, the simple $A_p(E,E)$-modules are given as follows.\\
{\rm(1)} $S(E,Q,U_i)$ $(0 \leq i \leq p-2)$,  
$$\dim S(E,Q,U_i)=\left\{
\begin{array}{ll}
p+1 & (i=0) \\
i+1 &(1\leq i\leq p-2).
\end{array}\right.$$
{\rm(2)}  $S(E,A,S(A)^{p-1}\ot {\det}^q)$ $(0 \leq q \leq p-2)$, 
$$\dim S(E,A,S(A)^{p-1}\ot {\det}^q)=p+1.$$
{\rm(3)} Simple $\Fp\Out(E)$-modules, namely, 
$$S(E,E,S^i\ot {\det}^i) ~(0\leq i\leq p-1,~0 \leq q \leq p-2).$$
{\rm(4)} The one dimensional module with trivial minimal 
subgroup. 
\end{prop}

Recall that  $M_i=CS^i+T^i$ and $N_i=\op_{0 \leq j \leq p-1}C^jM_i$ 
for $0 \leq i \leq p-2$, where $S^0=\Fp$ and $T^0=S^{p-1}$.

Let $S$ be a simple $A_p(E,E)$-module and let $e$ be an 
idempotent in $A_p(E,E)$ such that $Se=S$ and $S'e=0$ 
for any simple $A_p(E,E)$-module $S'$ which is not isomorphic to 
$S$. We shall obtain an $\Fp$-subspace $W$ of $H^*(E)$ 
such that the multiplication by $e$ induces isomorphism 
$$W \isoarr We =H^*(E)e.$$
In general, $W$ is not necessarily an $A_p(E,E)$-submodule. 
But for simple modules with minimal subgroup $A$, 
we see that $W$ is an $A_p(E,E)$-submodule and the equality 
$W=We$ holds. 

By Proposition \ref{p82} and \ref{p84}, we have the following.

\begin{thm}\label{t102}
Let $Q \leq E$, $|Q|=p$ and $U_i~(0 \leq i \leq p-2)$ be simple   
$kQ$-modules. Let $e$ be an idempotent corresponding to 
the simple $A_p(E,E)$-module $S(E,Q,U_i)$. Then 
$$W \isoarr We=H^*(E)e$$ 
where 
$$W=\left\{
\begin{array}{lc}
\Fp[C]\{M_0\}=\Fp[C]\{\Fp C+S^{p-1}\}& (i=0) \\
\Fp[C]\{S^i\} & (1 \leq i \leq p-2).
\end{array}\right.$$
\end{thm}

\begin{proof}
Since we have the following decompositions into $\Fp$-subspaces, $$\bCA\{M_0\}=\Fp[C]\{M_0\}\op \bCA\{VM_0\}$$ 
and 
$$Z_i=
\Fp[C]\{S^i\}\op \bCA\{v^iM_0+VM_i\}$$ 
the result follows from Proposition \ref{p82} and \ref{p84}.
\end{proof}

Similarly, by Corollary \ref{c93}, we have the following. 
  
\begin{thm}\label{t103}
Let $e$ be an idempotent in $A_p(E,E)$ corresponding  to the 
simple $A_p(E,E)$- module $S(E,A,S^{p-1}\ot {\det}^q)$. 
Then 
$$H^*(E)e=
\left\{
\begin{array}{cc}
\bDA \{D_2N_0\} & (q=0) \\
\bDA\{v^qN_q\} & (1 \leq q \leq p-2).
\end{array}\right.$$
\end{thm}

Next, we consider simple $A_p(E,E)$-modules which corresponds 
to simple $\Fp\Out(E)$-modules.  
First, we consider simple modules $\det^q$ and 
$S^{p-1} \ot \det^q$. 

\begin{thm}\label{t104} 
{\rm(1)} Let $e$ be an idempotent corresponding to the trivial 
$A_p(E,E)$-module. Then 
$$H^*(E)e=\bDA^+e \simeq \bDA^+.$$
{\rm(2)} Let $1 \leq q \leq p-2$. Then there exists an 
idempotent $e$ corresponding to the simple $A_p(E,E)$-module $S(E,E,{\det}^q)$ such that 
$$H^*(E)e=\bCA\{v^q\}e \simeq \bCA\{v^q\}.$$
{\rm(3)} Let $e$ be an idempotent corresponding to the 
simple $A_p(E,E)$-module $S(E,E,S^{p-1})$. 
Then 
$$H^*(E)e=\bDA\{VS^{p-1}\}e \simeq \bDA\{VS^{p-1}\}.$$
{\rm(4)} Let $1 \leq q \leq p-2$. Then there exists an 
idempotent $e$ corresponding to the 
simple $A_p(E,E)$-module $S(E,E,S^{p-1}\ot {\det}^q)$ 
such that 
$$H^*(E)e=\bCA\{v^q S^{p-1}\}e \simeq \bCA\{v ^q S^{p-1}\}.$$
\end{thm}

\begin{proof}
(1) Let $e$ be an idempotent corresponding to the trivial 
$A_p(E,E)$-module. 
Since direct summands in (\ref{t74}.2), (\ref{t74}.3) and 
(\ref{t74}.4) have no 
simple subquotient module which is isomorphic to 
$S(E,E,\Fp)$ by Corollary \ref{c710}, \ref{c712} and \ref{c714}, 
it follows that  $H^*(E)e=\bCA \{\Fp+S^{p-1}\}e$. 
By Corollary \ref{c77} and Lemma \ref{l21}, we have
$$H^*(E)e
=\Fp[D_1]^+e \op \bDA\{D_2\}e
 \simeq \Fp[D_1]^+ \op \bDA\{D_2\}
=\bDA^+.$$
(2) If $S$ is a simple $A_p(E,E)$-module such that 
$S$ has a composition factor isomorphic to $\det^q$ as a 
$\GL_2(\bF_p)$-module, then $S \simeq S(E,E, {\det}^q)$. 
Hence if  
$e$ is an idempotent in $\Fp\GL_2(\bF_p)$ corresponding to 
the simple $\Fp\GL_2(\bF_p)$-module $\det^q$, then 
$e$ is an idempotent corresponding 
to the simple $A_p(E,E)$-module $S(E,E,{\det}^q)$.
It follows that 
$$H^*(E)e=H_{0,q}=\bCA\{v^q\}  $$ by Corollary \ref{c45}. 
\\
(3) Let $e$ be an idempotent corresponding to the 
simple $A_p(E,E)$-module $S(E,E,S^{p-1})$.  
Then, by Corollary \ref{c77}, \ref{c710}, \ref{c712}, \ref{c714} 
and Lemma \ref{l21}, we have 
$$H^*(E)e=\bDA \{VS^{p-1}\}e = \bDA\{VS^{p-1}\}.$$
(4) As in the proof of (2), if $e$ is 
an idempotent in $\Fp\GL_2(\Fp)$ corresponding to 
$S^{p-1}\ot \det^q$, then 
$e$ is an idempotent corresponding to 
$S(E,E,S^{p-1}\ot {\det}^q)$. 
Hence it follows that 
$$H^*(E)e=H_{p-1,q}=\bCA\{v^qS^{p-1}\}.$$
by Corollary \ref{c45}. 
\end{proof}

Let us consider remaining simple $\Fp\Out(E)$-modules, that is, 
$S^i \ot \det^q$ for $1 \leq i \leq p-2$, $0 \leq q \leq p-2$. 

\begin{thm}\label{t105}
Let $1\leq i \leq p-2$ and $0 \leq q \leq p-2$. 
Let 
$$S=S^iv^q, \quad T=T^{p-i-1}v^{s}$$
where $s \equiv i+q \pmod{p-1}$, $0\leq s \leq p-2$. 
Then there exists an idempotent $e$ in $A_p(E,E)$ which corresponds to $S(E,E,S^i \ot \det^q)$ such that  
$H^*(E)e=We \simeq W$ for  
the following $\Fp$-subspace $W$: 
$$\begin{array}{lcll}
\bCA\{VS\} &\op& \bDA\{VT\} & (q \equiv 2i \equiv 0) \\
\bCA\{VS\} &\op& \bCA\{T\}  & (q \equiv 0,~2i \not\equiv 0) \\
\bDA\{S\}  &\op& \bDA\{VT\} & (i=q,~ 3i \equiv 0) \\
\bDA\{S\}  &\op& \bCA\{T\}  & (i=q,~ 3i \not\equiv 0)\\
\bCA\{S\}  &\op& \bDA\{VT\} & (q \ne 0,~i \ne q,~q+2i \equiv 0) \\
\bCA\{S\}  &\op& \bCA\{T\}  & (q \ne 0,~i \ne q,~q+2i \not\equiv 0) 
\end{array}$$
where $\equiv$ means equivalent modulo $p-1$.
\end{thm}

\begin{proof}
Let $f$ be an idempotent in $\Fp\GL_2(\Fp)$ which 
corresponds to the simple $\Fp\GL_2(\Fp)$-module 
$S^i \ot \det^q$.  Since 
$$S(E,E,S^i \ot {\det}^q) \simeq S^i \ot {\det}^q$$
as $\Fp\GL_2(\Fp)$-modules, there exists an idempotent $e$ 
in $A_p(E,E)$ which corresponds to the simple 
$A_p(E,E)$-module $S(E,E,S^i \ot {\det}^q)$ such that 
$fe=ef=e$. 
Note that 
$S(E,E,L)(f-e)=0$ for any simple $\Fp\GL_2(\Fp)$-module 
$L$. 
Since 
$$H^*(E)f=\bCA\{S\}f \op \bCA\{T\}f \simeq 
\bCA\{S\} \op \bCA\{T\}$$
by Corollary \ref{c45}, it follows that 
$$H^*(E)e=\bCA\{S\}e \op \bCA\{T\}e.$$ 

First, we consider $\bCA\{S\}e$. 
Assume $q=0$. Then 
$$\bCA\{S\} \subset \Fp[V]\{VS\}+Z_i$$
where $Z_i=\Fp[C]\{S\}+\bCA\{v^iM_0+VM_i\}$ as in 
Proposition \ref{p79}. By Corollary \ref{c710} and Lemma \ref{l21}, 
we have 
\begin{eqnarray*}
(\Fp[V]\{VS\}+Z_i)e&=&
\Fp[V]\{VS\}e+\bCA\{v^iM_0+VM_i\}e \\
&\simeq& \Fp[V]\{VS\}+\bCA\{v^iM_0+VM_i\}e.
\end{eqnarray*}
Since every composition factor of $\bCA\{v^iM_0+VM_i\}$ 
has minimal subgroup $E$ by Corollary \ref{c710}, 
it follows that 
\begin{eqnarray*}
\bCA\{v^iM_0+VM_i\}e&=&\bCA\{v^iM_0+VM_i\}f \\
&=&\bCA\{CVS\}f = \bCA\{CVS\}.
\end{eqnarray*}
Hence, if $q=0$, then 
\begin{eqnarray*}
\bCA\{S\}e &\subset& \Fp[V]\{VS\}e \op \bCA\{CVS\} \\
&=& \bCA\{VS\}e \simeq \bCA\{VS\}  
\end{eqnarray*}
and 
$$\bCA\{S\}e=\bCA\{VS\}e \simeq \bCA\{VS\}.$$

Assume $i=q$. Then 
$$\bCA\{S\}=\bDA\{\sum_{j=0}^{p-1}C^{j+1}S\} \op 
\bDA\{S\}$$
and we have 
$$\bDA\{\sum_{j=0}^{p-1}C^{j+1}S\}\subset \bDA\{v^iN_i\}$$
$$\bDA\{S\} \subset \bDA\{v^i(S^i+VT^i)\}.$$
Since $\bDA\{v^iN_i\}e=\bDA\{v^i(S^i+VT^i)\}(f-e)=0$ 
by Corollary \ref{c712} and $\bCA\{S\}f=\bCA\{S\}$, 
we have 
$$\bCA\{S\}e=\bDA\{S\}e =\bDA\{S\}.$$

Assume $q \ne 0$ and $i \ne q$. Then 
$\bCA\{S\}(f-e)=0$ by Corollary \ref{c714}. Hence 
we have 
$$\bCA\{S\}e=\bCA\{S\}f =\bCA\{S\}.$$
 
Next, we consider $\bCA\{T\}e$. 
Assume $q+2i \equiv 0$. Then 
$$\bCA\{T\}=\bDA\{\sum_{j=0}^{p-1}C^jT\} \op \bDA\{VT\}$$
and we have  
$$\bDA\{\sum_{j=0}^{p-1}C^jT\} \subset 
\bDA\{v^{p-1-i}N_{p-1-i}\}$$
$$\bDA\{VT\} \subset 
\bDA\{v^{p-1-i}(S^{p-1-i}+VT^{p-1-i})\}$$
since $s=p-1-i$. 
Since 
$$\bDA\{v^{p-1-i}N_{p-1-i}\}e=\bDA\{v^{p-1-i}(S^{p-1-i}+VT^{p-1-i})\}(f-e)=0$$ by Corollary \ref{c712} and 
$\bCA\{T\}f \simeq \bCA\{T\}$,  
we have 
$$\bCA\{T\}e=\bDA\{VT\}e \simeq \bDA\{VT\}.$$

Assume $q+2i \not\equiv 0$ and $i+q \equiv 0$. 
Then by Corollary \ref{c710} and Lemma 
\ref{l21}, 
\begin{eqnarray*}
(\Fp[C]\{T\}+Z_{p-1-i})e &=& 
\Fp[C]\{T\}e \op \bCA\{v^{p-1-i}M_0+VM_{p-1-i}\}e \\
&=& \Fp[C]\{T\}e \op \bCA\{VT\}e \\
&\simeq & \Fp[C]\{T\} \op \bCA\{VT\}. 
\end{eqnarray*}
Hence we have 
$$\bCA\{T\}e \simeq \bCA\{T\}.$$

Finally, assume $q+2i \not\equiv 0$ and 
$i+q \not\equiv 0$. Then  
$$\bCA\{T\} \subset \sum_{1 \leq j \ne r \leq p-2}
\bCA\{v^r(S^j+T^j)\}$$  
since $p-1-i \ne s$.
Hence $\bCA\{T\}(f-e)=0$ by Corollary \ref{c714} and 
in particular, we have 
$$\bCA\{T\}e=\bCA\{T\}f \simeq \bCA\{T\}.$$
This completes the proof.   
\end{proof}

By these results, we can describe the cohomology of 
stable summands of $BE$.   
\begin{defin}\label{d106} 
\begin{rm}Let $X_M$ be an indecomposable 
stable summand of $BE$ corresponding to a simple right 
$A_p(E,E)$-module $M$.
Let $X(M)$ be the sum of indecomposable summands equivalent 
to $X_M$ in the complete stable splitting of $BE$. Let 
$e$ be an idempotent in $A_p(E,E)$ corresponding to $M$, namely, 
$M=Me$ and $M'e=0$ for a simple $A_p(E,E)$-module which is 
not isomorphic to $M$. Let 
$$H^*(X(M))=H^*(E)e.$$
   
Let $X_{i,q}=X_{S(E,E,S^i \ot {\det}^q)}$ and 
$X(i,q)=X(S(E,E,S^i \ot {\det}^q))$ for $0 \leq i \leq p-1$ 
and $0\leq q \leq p-2$. 
Note that $X(i,q)\sim (i+1)X_{i,q}$. 
Moreover we write 
$$X(E,A,q)=X(S(E,A,S^{p-1}\ot {\det}^q))$$
and 
$$X(E,Q,q)=X(S(E,Q,U_q)).$$
\end{rm}
\end{defin}

\begin{cor}\label{c107} We have the following isomorphisms: 
\begin{eqnarray*}
H^*(X(0,0)) &\simeq& \bDA^+  \\
H^*(X(0,q)) &\simeq& \bCA\{v^q\}  \\
H^*(X(p-1,0)) &\simeq& \bDA\{VS^{p-1}\}  \\
H^*(X(p-1,q)) &\simeq& \bCA\{v ^q S^{p-1}\}~
(1 \leq q \leq p-2).
\end{eqnarray*}
\end{cor}

\begin{proof}
This follows from Theorem \ref{t104}.
\end{proof}

\begin{cor}\label{c108} 
Let $1\leq i \leq p-2$ and $0 \leq q \leq p-2$. Let 
$$S=S^i v^q,\quad T=T^{p-i-1} v^s$$
where $s \equiv i+q \pmod{p-1}$, $0\leq s \leq p-2$.
Then $H^*(X(i,q))$ is isomorphic to the following subspace:
$$\begin{array}{lcll}
\bCA\{VS\} &\op& \bDA\{VT\} & (q \equiv 2i \equiv 0) \\
\bCA\{VS\} &\op& \bCA\{T\}  & (q \equiv 0,~2i \not\equiv 0) \\
\bDA\{S\} &\op&   \bDA\{VT\} & (i=q,~ 3i \equiv 0) \\
\bDA\{S\}  &\op&  \bCA\{T\}  & (i=q,~ 3i \not\equiv 0)\\
\bCA\{S\}  &\op&  \bDA\{VT\} & (q \ne 0,~i \ne q,~q+2i \equiv 0) \\
\bCA\{S\}  &\op&  \bCA\{T\}  & (q \ne 0,~i \ne q,~q+2i \not\equiv 0) 
\end{array}$$
where $\equiv$ means equivalent modulo $p-1$.
\end{cor}

\begin{proof}
This follows from Theorem \ref{t105}. 
\end{proof}

\begin{cor}\label{c109}
We have the following isomorphisms:
$$H^*(X(E,A,q))\simeq 
\left\{
\begin{array}{cc}
\bDA\{D_2(\bigoplus_{j=0}^{p-1}C^j(\Fp C+S^{p-1}))\} & (q=0)\\
\bDA\{v^q(\bigoplus_{j=0}^{p-1}C^j(CS^q+T^q))\} 
& (1 \leq q \leq p-2) 
\end{array}\right. $$
$$H^*(X(E,Q,i)) \simeq 
\left\{
\begin{array}{cc}
\Fp[C]\{\Fp C+S^{p-1}\} & (i=0)\\
\Fp[C]\{S^i\} & (1 \leq i \leq p-2). 
\end{array}\right.$$
\end{cor}

\begin{proof}
This follows from Theorem \ref{t102} and \ref{t103}. 
\end{proof}



\begin{thebibliography}{00}
%
%
\bibitem{B}D.~J.~Benson, 
Stably splitting $BG$, 
Bull. Amer. Math. Soc. 33 (1996), 189-198.
%
\bibitem{BF}D.~J.~Benson and M.~Feshbach, 
Stable splittings of classifying spaces of finite groups, 
Topology 31 (1992), 157-176. 
%
\bibitem{Bo}S.~ Bouc, Foncteurs d'ensembles munis d'une double action, 
J. Algebra 183 (1996), 664-736. 
%
\bibitem{BST}S.~ Bouc, R.~Stancu and J.~Th\'evenaz, 
Simple biset functors and double Burnside ring, 
arXiv: 1203.0195.
%
\bibitem{C}G.~Carlsson, Equivariant stable homotopy  
and Segal's Burnside ring conjecture, Ann. Math. 120 (1984), 
189-224. 
%
\bibitem{DP}J.~Dietz and  S.~Priddy, 
The stable homotopy type of rank two $p$-groups, 
{\it Homotopy theory and its applications}, 
Contemp. Math. 188, Amer. Math. Soc., Providence, RI, (1995), 
93-103.
%
\bibitem{HK}J.~C.~Harris and N.~J.~Kuhn, 
Stable decompositions of classifying spaces of finite abelian 
$p$-groups, Math. Proc. Camb. Phil. Soc. 103 (1988), 427-449.
%
\bibitem{Gl}D.~J.~Glover, 
A study of certain modular representations, 
J. Algebra 51 (1978), 425-475.
%
\bibitem{Gr}D.~J.~Green, On the cohomology of the sporadic 
simple group $J_4$, Math. Proc. Camb. Phil. Soc. 113 
(1993), 253-266.
%
\bibitem{L91}I.~J.~Leary, 
The integral cohomology rings of some $p$-groups, 
Math. Proc. Camb. Phil. Soc. 110 (1991), 25-32.
%
\bibitem{L92}I.~J.~Leary, 
The mod-$p$ cohomology rings of some $p$-groups, 
Math. Proc. Camb. Phil. Soc. 112 (1992), 63-75.
%
\bibitem{Lewis}G.~Lewis, The integral cohomology rings 
of groups of order $p^3$, 
Trans. Amer. Math. Soc. 132 (1968), 501-529.
%
\bibitem{LMM}L.~G.~Lewis, J.~P.~May and J.~E.~ McClure, 
Classifying $G$-spaces and the Segal conjecture, 
{\it Current trends in algebraic topology}, CMS Conf. Proc. 2 
(1982), 165-179.
%
\bibitem{MP}J.~Martino and S.~Priddy, 
The complete stable splitting for the classifying space 
of a finite group, Topology 31 (1992), 143-156.
%
\bibitem{Q}D.~Quillen, 
The spectrum of an equivariant cohomology ring: I, 
Ann. of Math. 94 (1971), 549-572.
%
\bibitem{RV}A.~Ruiz and A.~Viruel, 
The classification of $p$-local finite groups 
over the extraspecial group of order $p^3$ and 
exponent $p$, Math. Z. 248 (2004), 45-65.
%
\bibitem{S}H.~Sasaki, 
Mod $p$ cohomology algebras of finite groups with 
extraspecial Sylow $p$-subgroups, 
Hokkaido Math. J. 29 (2000), 263-302.
%
\bibitem{TY}M.~Tezuka and N.~Yagita, 
On odd prime components of cohomologies of sporadic simple groups 
and the rings of universal stable elements, 
J. Algebra 183 (1996), 483-513.
%
\bibitem{W}P.~Webb, Two classifications of simple 
Mackey functors with applications to 
group cohomology and decomposition of classifying spaces, 
J. Pure Appl. Algebra 88 (1993), 265-304.
%
\bibitem{Y98}N.~Yagita, 
On odd degree parts of cohomology of sporadic simple groups 
whose Sylow $p$-subgroup is the extra-special $p$-group 
of order $p^3$, J. Algebra 201 (1998), 373-391.
%
\bibitem{Y07}N.~Yagita, 
Stable splitting and cohomology of $p$-local finite groups 
over the extraspecial $p$-group of order $p^3$ and exponent $p$, 
Geometry and Topology Monographs 11 (2007), 399-434.
%
\end{thebibliography}
\end{document}